\documentclass{article}
\usepackage[utf8]{inputenc}
\usepackage{amsmath}
\usepackage{amsfonts}
\usepackage{amssymb}
\usepackage{amsthm}
\usepackage[all]{xy}
\usepackage{enumerate}
\usepackage{cite}
\usepackage[affil-it]{authblk}
\usepackage{hyperref}
\usepackage{bbm}
\hypersetup{hypertex=true,
            linkcolor=blue,
            anchorcolor=blue,
            citecolor=blue}

\theoremstyle{plain}
  \newtheorem{thm}{Theorem}[section]
  \newtheorem{lem}[thm]{Lemma}
  \newtheorem{prop}[thm]{Proposition}
  \newtheorem{cor}[thm]{Corollary} 
\theoremstyle{definition}
  \newtheorem{defn}[thm]{Definition}
  \newtheorem{rmk}[thm]{Remark}
  \newtheorem{ex}[thm]{Example}

\theoremstyle{plain}

\DeclareMathOperator{\im}{im}
  
\numberwithin{equation}{section}

\allowdisplaybreaks

\title{An upper bound for the Lusternik-Schnirelmann category of relative Sullivan algebras}
\author{Jiawei Zhou}
\date{\today}

\begin{document}

\maketitle

\begin{abstract}
This paper addresses a question posed by F\'elix, Halperin and Thomas. We prove that the Lusternik-Schnirelmann category of a relative Sullivan algebra is finite whenever the corresponding categories of both the base and fiber algebras are finite. Furthermore, we provide a corresponding upper bound for the Toomer invariant.
\end{abstract}

\section{Introduction}

The Lusternik-Schnirelmann category (LS category) of a topological space $X$, denoted by $\text{cat}(X)$, is the least non-negative integer $m$ such that $X$ can be covered by $m+1$ open subsets, each categorical (contractible in $X$). This invariant was originally introduced in \cite{LS} to estimate the number of critical points of smooth functions on a manifold.

While the definition is straightforward, computing the LS category is generally difficult. To estimate it, several other invariants have been introduced, such as the cone length, the Toomer invariant, and the cup length. In rational homotopy theory, these invariants admit rational counterparts. When $X$ is simply connected and has finite type rational homology, these rational invariants are defined as the minimum value over all spaces rationally homotopy equivalent to $X$. Such equivalences are given by continuous maps inducing isomorphisms in rational homology, equivalently (by the Whitehead–Serre Theorem) in rational homotopy groups.

However, when $X$ is not simply connected, the equivalence in the Whitehead–Serre theorem no longer holds in general, although it remains valid for certain special spaces, such as rationally nilpotent spaces. For arbitrary connected spaces, rational homology equivalences need not induce isomorphisms on rational higher homotopy groups. We therefore formulate our results using the algebraic category defined via Sullivan models.

Sullivan \cite{sullivan} employed commutative differential graded algebras (CDGAs) as algebraic models for connected spaces. Given a connected space $X$, any CDGA quasi-isomorphic to $A_{PL}(X)$, the CDGA of polynomial differential forms on $X$ is called a model of $X$. For instance, when the ground field is $\mathbb{R}$ and $X$ is a smooth manifold, the CDGA $\Omega^*(X)$ of differential forms is a model of $X$. Each quasi-isomorphism class contains a special type of CDGA called a Sullivan algebra, and a Sullivan algebra that is a model of $X$ is called a Sullivan model of $X$.

One can define the algebraic version of LS category via Sullivan algebras and then extend it to general CDGAs by quasi-isomorphism. If a Sullivan algebra is a model of a space $X$, its LS category is defined to be the rational LS category of $X$. When $X$ is simply connected with finite type rational cohomology, this algebraic definition agrees with the topological definition of the rational LS category.

Many fibrations admit algebraic models in terms of relative Sullivan algebras. Consider a fibration $F\to E\to B$. If $\Lambda Z$ and $\Lambda Z\otimes\Lambda W$ are Sullivan models of $B$ and $E$, respectively, then $\Lambda W$ is a Sullivan model of $F$ under suitable conditions, for instance, when $F$ and $B$ are simply connected and either $H^*(F,\mathbb{Q})$ or $H^*(B,\mathbb{Q})$ is of finite type. When these conditions fail, however, the Sullivan model of $F$ may differ substantially from $\Lambda W$.
 
It is easy to see that $\text{cat}(E)\leq(\text{cat}(F)+1)(\text{cat}(B)+1)-1$. Consequently, in their book \cite{FHT2}, F\'elix, Halperin and Thomas posed the question of whether the LS category of $\Lambda Z\otimes\Lambda W$ can be bounded similarly in terms of those of $\Lambda Z$ and $\Lambda W$. We give a positive answer.

\begin{thm}\label{estimate cat intro}
Let $\Lambda Z\otimes\Lambda W$ be a relative Sullivan algebra. If ${\rm cat}(\Lambda Z)=m$ and ${\rm cat}(\Lambda W)=n$, then
$$
{\rm cat}(\Lambda Z\otimes\Lambda W) \leq (m+1)(n+1)-1.
$$
Moreover, if $\Lambda Z\otimes\Lambda W$ itself is a minimal Sullivan algebra and $n\geq 1$, then
$$
{\rm cat}(\Lambda Z\otimes\Lambda W) \leq (m+1)n.
$$
\end{thm}

The proof is formulated in terms of module category mcat, the
module-theoretic analogue of the LS category initially introduced by Halperin and Lemaire \cite{HL}. Hess' theorem \cite{hess} shows that the LS category and the module category of a minimal Sullivan algebra are the same. Thus, Theorem \ref{estimate cat intro} follows from Theorem \ref{main} and Theorem \ref{main for minimal}, both of which concern module categories.

A minimal Sullivan algebra $\Lambda V$ can be regarded as a relative Sullivan algebra with base $\Lambda V^1$, generated by degree-1 elements of $V$, and fiber $(\Lambda V^{\geq 2},\bar{d})$, generated by the elements of $V$ in degrees at least 2 and equipped with the induced differential. Consequently, the corollary below follows from Theorem \ref{estimate cat intro}, addressing the first part of the question posed by F\'elix, Halperin and Thomas.

\begin{cor}\label{cat for minimal intro}
Let $\Lambda V$ be a minimal Sullivan algebra with $\dim V^i<\infty$ for all $i\geq 2$. Then ${\rm cat}(\Lambda V)<\infty$ if and only if both ${\rm }cat(\Lambda V^1)$ and ${\rm cat}(\Lambda V^{\geq 2})$ are finite.
\end{cor}

In the same spirit, we can also obtain an upper bound for the Toomer invariant. Theorem \ref{estimate e} gives the corresponding statement for modules.

\begin{thm}\label{estimate e intro}
Let $\Lambda V=\Lambda Z\otimes\Lambda W$ be a relative Sullivan algebra. Suppose that $e_{\Lambda Z}(\Lambda V)=m$ and that there exists a fixed integer $n\geq 0$ such that the inclusion
$$
\Lambda^{\geq q+n+1} W \to \Lambda^{\geq q} W
$$
induces the zero map on cohomology for all $q$. Then
$$
e_{\Lambda V}(\Lambda V) \leq (m+1)(n+2)-2.
$$
If, in addition, either of the following conditions holds:

(i) $\Lambda V$ is a minimal Sullivan algebra;

(ii) cat$(\Lambda W)=n$;

\noindent then
$$
e_{\Lambda V}(\Lambda V) \leq (m+1)(n+1)-1.
$$
When both conditions hold and $n\geq 1$, we have the sharper bound
$$
e_{\Lambda V}(\Lambda V) \leq (m+1)n.
$$
\end{thm}

This paper is organized as follows. In Section 2, we review the basic definitions and previous results. To make the paper self-contained, we provide proofs of certain consequences that are not covered in \cite{FHT} and \cite{FHT2}. The fact that DG modules form a model category is not directly used, although some statements follow from it immediately. In Section 3, we prove the main results: Theorems \ref{main}, \ref{main for minimal} and \ref{estimate e}, which lead to the conclusions outlined in Theorem \ref{estimate cat intro}, Corollary \ref{cat for minimal intro} and Theorem \ref{estimate e intro}.

The author would like to thank Ruizhi Huang and Li-Sheng Tseng for helpful discussions. We are grateful to the referee for the careful review, which identified several issues and offered constructive suggestions for improvement. This research was supported by the National Key Research and Development Program of China No. 2020YFA0713000 and Jiangxi Provincial Natural Science Foundation No. 20262BAC240204.

\section{Preliminaries}
\subsection{Notation and Conventions}
In this paper, our notation and conventions generally follow \cite{Halperin}, \cite{FHT} and \cite{FHT2}.

As in the introduction, we use CDGA for commutative differential graded algebra, DG module for differential graded module, and LS category for Lusternik-Schnirelmann category.

The ground field $\mathbbm{k}$ has characteristic zero. All CDGAs are assumed non-negatively graded, while DG modules are allowed to be $\mathbb{Z}$-graded. A morphism of DG modules over a CDGA $A$ will be referred to simply as an $A$-module morphism or a morphism. If the map is only a morphism of underlying graded modules, this will be explicitly indicated.

We shall use $\cong$ to denote isomorphisms, $\simeq$ to denote quasi-isomorphisms, and $\sim$ to denote homotopic DG module morphisms. The definition of the latter will be given in this section.

Let $(\Lambda V,d)$ be a CDGA whose underlying graded algebra is freely generated by a nonnegatively graded vector space $V$. We call it a \textbf{KS complex}, short for Koszul-Sullivan complex, if it satisfies the Sullivan condition, i.e., $V$ is the union of an increasing filtration
$$
0=V(-1) \subset V(0) \subset V(1) \subset \ldots \subset V(k) \subset \ldots
$$
such that $dV(k)\subset\Lambda V(k-1)$. When the differential is clear from context, we will abbreviate $(\Lambda V,d)$ to $\Lambda V$.

For a homogeneous element $v\in V$, we denote its degree by $|v|$. The subspace of all $v\in V$ with $|v|=n$ will be written as $V^n$. We also use $\Lambda^m V$ to represent the subspace of $\Lambda V$ consisting of elements with word length $m$, and use $\Lambda V^n$ to denote the graded algebra generated by $V^n$. Note that the latter is distinct from the subspace of $\Lambda V$ concentrated in degree $n$, which is usually written as $(\Lambda V)^n$. For the subspace of $V$ spanned by elements with degree less than $n$, we use $V^{<n}$. Analogous notations, such as $\Lambda^{\leq m}V,\Lambda V^{\geq n}$ and so on, will also be employed.

When $V^0$ is non-trivial, we assume the existence of an augmentation $\epsilon:(\Lambda V,d)\to(\mathbbm{k},0)$ sending $V$ to 0. This augmentation, which is necessary for Theorem 2.2 of \cite{Halperin}, is only used in Remark \ref{non-connected remark} and in the proof of Theorem \ref{mcat over non-minimal}.

We mainly focus on the case $V=V^{\geq 1}$. A KS complex generated by such a $V$ is referred to as a \textbf{Sullivan algebra}.

A KS complex $\Lambda V$ is termed \textbf{minimal} if it also satisfies $dV^n \subset \Lambda V^{\leq n}$. If $\Lambda V$ is a Sullivan algebra, this condition is equivalent to $dV\subset\Lambda^{\geq 2}V$. It is a well-known result that given any connected CDGA $A$ (i.e., $H^0(A)=\mathbbm{k}$), there exists a quasi-isomorphism from a Sullivan algebra $\Lambda V$ to $A$. Such a $\Lambda V$ is called a \textbf{Sullivan model} of $A$. Furthermore, $A$ admits a Sullivan model which is minimal, unique up to isomorphism. In this case, this model is called the \textbf{minimal Sullivan model} or the \textbf{minimal model} of $A$.

Let $B$ be a connected CDGA. An inclusion $B\to B\otimes\Lambda W$ is termed a \textbf{$\Lambda$-extension} if it is a CDGA morphism satisfying the following condition: $W$ is the union of an increasing filtration
$$
0=W(-1) \subset W(0) \subset W(1) \subset \ldots \subset W(k) \subset \ldots
$$
such that $d(1\otimes W(k)) \subset (\mathbbm{k}\oplus B^{\geq 1})\otimes\Lambda W(k-1)$. Such a $\Lambda$-extension induces a KS complex $(\Lambda W,\bar{d})$ where $\bar{d}$ is the $\mathbbm{k}\otimes\Lambda W$ component of $d$. This KS complex $(\Lambda W,\bar{d})$ is called the \textbf{fiber} of the $\Lambda$-extension, and the CDGA $B$ is called the \textbf{base}.

If, in addition to the Sullivan filtration condition, the $\Lambda$-extension satisfies $d(1\otimes W^n) \subset B\otimes\Lambda W^{\leq n}$, then it is referred to as \textbf{minimal}.

When the $\Lambda$-extension has $W$ concentrated in positive degrees, it is called a \textbf{Sullivan extension}, and the CDGA $B\otimes\Lambda W$ is then termed a \textbf{relative Sullivan algebra}.

\subsection{Semifree Resolutions}
\begin{defn}
Let $P$ be a DG module over a CDGA $A$. We call $P$ \textbf{semifree} if it can be written as the union of an increasing filtration
$$
0=P(-1) \subset P(0) \subset P(1) \subset \ldots \subset P(k) \subset \ldots
$$
by sub-$A$-modules, such that $P(k)/P(k-1)\cong A\otimes W(k)$ with $dW(k)=0$, i.e., $P(k)/P(k-1)$ is a free DG module over $A$.
\end{defn}

Since each quotient $P(k)/P(k-1)$ is a free, the exact sequence
$$
0 \to P(k-1) \to P(k) \to P(k)/P(k-1) \to 0
$$
splits as graded $A$-modules. By induction, it follows that $P$ can be written as $A\otimes V$ as a graded module, where $V=\bigcup_k V(k)$ for an increasing sequence of graded vector spaces $V(k)$ such that $dV(k)\subset A\otimes V(k-1)$.

\begin{defn}
Let $M$ be a DG module over a CDGA $A$. An \textbf{$A$-semifree resolution} of $M$ is a semifree $A$-module $P$ equipped with a quasi-isomorphism $P\stackrel{\simeq}{\longrightarrow} M$.
\end{defn}

According to Proposition 6.6 of \cite{FHT}, every DG module has a semifree resolution. Actually, the quasi-isomorphism in a semifree resolution can be chosen to be surjective.

\begin{prop}\label{surjective resolution}
Let $A$ be a CDGA. Every $A$-module $M$ has a semifree resolution $f:P\stackrel{\simeq}{\longrightarrow} M$ with $f$ surjective.
\end{prop}
\begin{proof}
Proposition 6.6 of \cite{FHT} gives a semifree resolution $f:P\stackrel{\simeq}{\longrightarrow} M$, with $P(0)$ identified as $A\otimes V(0)$, where $V(0)$ is the space of cocycles in $M$. Via the inclusion $V(0)\to P(0)$, we regard a chosen copy of each cocycle in $M$ as an element of $P(0)$. Note that $f|_{P(0)}$ maps this copy isomorphically onto $V(0)\subset M$.

Now take any $x\in M$. Since $dx$ is a cocycle in $M$, there exists a cocycle $\alpha\in P(0)$ such that $f(\alpha)=dx$. As $f$ is a quasi-isomorphism, $\alpha=d\beta$ for some $\beta\in P$. Then
$$
d(x-f(\beta))=dx-f(d\beta)=f(\alpha)-f(\alpha)=0.
$$
Hence, $x-f(\beta)$ is a cocycle in $M$, and therefore lies in $V(0)\subset\im f$. It follows that $x\in\im f$.
\end{proof}

\begin{defn}
Two DG module morphisms $f,g:M\to N$ over a CDGA $A$ are called \textbf{homotopic} if there exists a graded $A$-linear map $\theta:M\to N$ of degree $-1$ such that $f-g=d\theta+\theta d$. Here the $A$-linearity satisfies the Koszul sign rule. That is
$$
a\cdot\theta(x)=(-1)^{|a|}\theta(a\cdot x) \text{ for } a\in A, x\in M.
$$
We will write $f\sim g$ to indicate that $f$ and $g$ are homotopic.
\end{defn}

\begin{prop}\label{commutative lift}
Suppose $A$ is a CDGA, and $P$ is a semifree $A$-module. Let $\phi: P\to N$ and $f:M\stackrel{\simeq}{\longrightarrow} N$ be morphisms of DG $A$-modules, where $f$ is a quasi-isomorphism.

(i) (Proposition 6.4(ii) of \cite{FHT}) There exists a unique homotopy class of morphisms $\psi:P\to M$ such that $\phi \sim f\circ\psi$.

(ii) (Exercise 4 of Section 6 of \cite{FHT}) If $f$ is surjective, we may choose $\psi$ such that $\phi=f\circ\psi$. Such a morphism $\psi$ is called a \textbf{lift} of $\phi$ through $f$.
$$
\xymatrix{
& M \ar[d]_f^{\simeq} \\
P \ar[r]^{\phi} \ar@{-->}[ru]^{\psi} & N
}
$$
\end{prop}
\begin{proof}
(i) The proof is provided in \cite{FHT}.

(ii) Write $P$ as the union of $P(k)=A\otimes V(k)$ with $dV(k)\subset P(k-1)$. Suppose that $\psi$ has been constructed on $P(k-1)$. We extend it to $P(k)$.

Choose a linearly independent set $\{v_{\alpha}\}$ such that $V(k) = V(k-1)\oplus\text{span}\{v_{\alpha}\}$. By hypothesis $\psi(dv_{\alpha})$ is a well-defined cocycle in $M$, and
$$
f(\psi(dv_{\alpha})) = \phi(dv_{\alpha}) = d(\phi(v_{\alpha})).
$$
As $f$ is a quasi-isomorphism, $\psi(dv_{\alpha})$ is a coboundary in $M$. Choose some $m_{\alpha}\in M$ with $dm_{\alpha}=\psi(dv_{\alpha})$. Then
$$
d(\phi(v_{\alpha})-f(m_{\alpha})) = \phi(dv_{\alpha})-f(dm_{\alpha}) = (f\circ\psi)(dv_{\alpha})-f(\psi(dv_{\alpha})) = 0.
$$
Thus, $\phi(v_{\alpha})-f(m_{\alpha})$ is a cocycle. Again, since $f$ is a quasi-isomorphism, there exist some $n''_{\alpha}\in N$ and a cocycle $m'_{\alpha}\in M$ such that
$$
\phi(v_{\alpha})-f(m_{\alpha})=f(m'_{\alpha})+dn''_{\alpha}.
$$
As $f$ is surjective, we have $n''_{\alpha}=f(m''_{\alpha})$ for some $m''_{\alpha}\in M$.

Set
$$
\psi(v_{\alpha})=m_{\alpha}+m'_{\alpha}+dm''_{\alpha}.
$$
Then $d\psi(v_{\alpha})=dm_{\alpha}=\psi(dv_{\alpha})$ and $f(\psi(v_{\alpha}))=\phi(v_{\alpha})$. Hence $\psi$ extends to $P(k)$ as an $A$-module morphism.

By induction, we obtain a morphism $\psi$ defined on $P$ satisfying $\phi=f\circ\psi$.
\end{proof}

\subsection{Lusternik-Schnirelmann category}

\begin{defn}\label{def of cat}
Let $\Lambda V$ be a minimal Sullivan algebra. Its \textbf{LS category} $\text{cat}(\Lambda V)$ is the least integer $m\geq 0$ for which there exists a minimal relative $\Lambda$-extension $f: \Lambda V \to \Lambda V\otimes\Lambda W$, together with CDGA morphisms $g:\Lambda V\otimes\Lambda W\to\Lambda V$ and $\phi:\Lambda V\otimes\Lambda W\to\Lambda V/\Lambda^{>m} V$, such that the following diagram commutes:
$$
\xymatrix{
\Lambda V \ar[rd]^f \ar[rdd]_q \ar[rr]^{\text{id}_{\Lambda V}} & & \Lambda V\\
& \Lambda V\otimes\Lambda W \ar[ur]^g \ar[d]_{\phi}^{\simeq} \\
& \Lambda V/\Lambda^{>m} V
}
$$
Here $\phi$ is required to be a quasi-isomorphism, and $q$ denotes the natural projection.
\end{defn}

\begin{rmk}
As mentioned in \cite{FHT2}, the condition $gf=\text{id}_{\Lambda V}$ can be weakened to the requirement that $gf$ and $\text{id}_{\Lambda V}$ be homotopic as CDGA morphisms from Sullivan algebras.
\end{rmk}

Every connected CDGA has a unique minimal Sullivan model up to isomorphism. Its LS category may therefore be defined via this minimal Sullivan model.

\begin{defn}
Let $A$ be a connected CDGA with minimal Sullivan model $\Lambda V$. Then $\text{cat}(A)$ is defined as $\text{cat}(\Lambda V)$.
\end{defn}

\begin{defn}
Let $\Lambda V$ be a minimal Sullivan algebra and $P$ a $\Lambda V$-semifree module. The \textbf{module category} of $P$, $\text{mcat}_{\Lambda V}(P)$, is the least integer $m\geq 0$ for which there exists a semifree $\Lambda V$-module $Q$ together with $\Lambda V$-module morphisms
$$
f:P\to Q, \quad g:Q\to P, \quad \phi:Q\to P/(\Lambda^{>m} V\cdot P)
$$
satisfying the following conditions:
$$
\phi \text{ is a quasi-isomorphism}, \quad g\circ f \sim id_{P}, \quad \phi\circ f \sim q,
$$
where $q:P \to P/(\Lambda^{>m} V\cdot P)$ denotes the natural projection. In other words, the following diagram is homotopy commutative:
\begin{align}\label{mcat graph}
\xymatrix{
P \ar[rd]^f \ar[rdd]_q \ar[rr]^{\text{id}_P} & & P\\
& Q \ar[ur]^g \ar[d]_{\phi}^{\simeq} \\
& P/(\Lambda^{>m} V\cdot P)
}
\end{align}
When the ground CDGA $\Lambda V$ is clear from context, we write $\text{mcat}(P)$ for $\text{mcat}_{\Lambda V}(P)$.
\end{defn}

As shown in the appendix of \cite{HL}, the above homotopy-commutative diagram can be made strictly commutative. This also follows from Proposition \ref{replace homotopy} below, which will be used frequently in the proof of our main Theorem.

The advantage of the homotopy-commutative definition is that $P$ can be replaced by any quasi-isomorphic $\Lambda V$-semifree module. By Proposition \ref{commutative lift}, it follows readily that any two quasi-isomorphic semifree modules have the same module category. Consequently, we can define the module category for arbitrary DG modules via their semifree resolutions.

\begin{defn}
Let $\Lambda V$ be a minimal Sullivan algebra, $M$ a $\Lambda V$-module, and $P$ a $\Lambda V$-semifree resolution of $M$. Then $\text{mcat}_{\Lambda V}(M)$ is defined to be $\text{mcat}_{\Lambda V}(P)$.
\end{defn}

\begin{defn}\label{def of e}
Let $\Lambda V$ be a minimal Sullivan algebra and $P$ a $\Lambda V$-semifree module. The \textbf{Toomer invariant} of $P$, denoted $e_{\Lambda V}(P)$ or simply $e(P)$ when no ambiguity arises, is the least integer $m\geq 0$ for which the projection $P\to P/(\Lambda^{>m} V\cdot P)$ induces an injection on cohomology.

In particular, $e(\Lambda V)$ is the least integer $m\geq 0$ such that the morphism $H(\Lambda V)\to H(\Lambda V/\Lambda^{>m} V)$ induced by the natural projection is injective.

For an arbitrary $\Lambda V$-module $M$, $e(M)$ is defined to be the Toomer invariant of any semifree resolution of $M$.
\end{defn}

\begin{prop}[Proposition 9.3 of \cite{FHT2}]
Let $\Lambda V$ be a minimal Sullivan algebra and $M$ a $\Lambda V$-module. Then $e(M)\leq {\rm mcat}(M) \leq {\rm mcat}(\Lambda V)$.
\end{prop}

\begin{thm}[Hess, \cite{hess}; Theorem 9.4 of \cite{FHT2}]
Let $\Lambda V$ be a minimal Sullivan algebra. Then ${\rm cat}(\Lambda V)={\rm mcat}_{\Lambda V}(\Lambda V)$.
\end{thm}

\begin{prop}\label{replace homotopy}
Let $A$ be a CDGA and $P=A\otimes V$ a semifree $A$-module.

(i) Suppose that $f:P\to Q$ is an $A$-module morphism. Then $f$ admits a factorization $p\,\circ F$, where $p:\widetilde{Q}\stackrel{\simeq}{\longrightarrow} Q$ is a surjective quasi-isomorphism and $F:P\to\widetilde{Q}$ is injective.

(ii) Suppose additionally that there exist $A$-module morphisms $g:Q\to M$ and $h:P\to M$ such that $g\circ f\sim h$. Then there exist an $A$-module quasi-isomorphism $p:\widetilde{Q}\stackrel{\simeq}{\longrightarrow} Q$ and morphisms $F:P\to\widetilde{Q}$ and $G:\widetilde{Q}\to M$ such that $p\circ F=f$ and $G\circ F=h$. 

(iii) Under the hypothesis of (i), let $I\subset A$ be a differential ideal and $q:P\to P/(I\cdot P)$ the projection. Suppose further that there exist an $A$-module morphism $g:Q\to P$ with $g\circ f\sim {\rm id}_{P}$, and a quasi-isomorphism $\phi:Q\stackrel{\simeq}{\longrightarrow} P/(I\cdot P)$ such that $q\sim\phi\circ f$. Then $Q,f,g,\phi$ can be replaced by $\widetilde{Q}$, $F:P\to \widetilde{Q}$, $G:\widetilde{Q}\to P$, and $\Phi:\widetilde{Q}\stackrel{\simeq}{\longrightarrow} P/(I\cdot P)$, respectively, such that $\widetilde{Q}$ is quasi-isomorphic to $Q$, $q=\Phi\circ F$ and $G\circ F= {\rm id}_{P}$.
$$
\xymatrix{
P \ar[dr]^F \ar[rdd]_q \ar[rr]^{{\rm id}_P} & & P \\
& \widetilde{Q} \ar[ur]^G \ar[d]_{\Phi}^{\simeq} \\
& P/(I\cdot P)
}
$$.

(iv) In each of the above statements, $\widetilde{Q}$ can be taken to be semifree.
\end{prop}
\begin{proof}
(i) Set $\widetilde{Q}=P\oplus Q\oplus (A\otimes sV)$, where $(sV)^n=V^{n+1}$. The suspension $s:V\to sV$ induces a graded $A$-linear morphism $S:P\to A\otimes sV$ of degree $-1$, characterized by $S(a\cdot v)=(-1)^{|a|}a\otimes sv$ for $a\in A, v\in V$. Define the differential $D$ on $\widetilde{Q}$ by
$$
Dx=dx \ (x\in P), \quad D\alpha=d\alpha \ (\alpha\in Q), \quad D(sv)=v+f(v)-S(dv) \ (sv\in sV).
$$
A straightforward calculation shows that $D$ is well-defined.

Let $F:P\to\widetilde{Q}$ be the natural inclusion. Define $p:\widetilde{Q}\to Q$ by
$$
p|_P=f, \quad p|_Q=-\text{id}_Q, \quad p|_{A\otimes sV}=0.
$$
For $sv\in sV$ we have
$$
(p\circ D)(sv)=f(v)-(\text{id}_{Q}\circ f)(v)=0=(D\circ p)(sv).
$$
Thus, $p$ is a surjective morphism over DG $A$-modules. It remains to show that $p$ is a quasi-isomorphism.

Let $x+\alpha+\sum a_i\otimes sv_i$ be a cocycle in $\widetilde{Q}$ with $x\in P, \alpha\in Q, a_i\in A$ and $sv_i\in sV$. Suppose $p(x+\alpha+\sum a_i\otimes sv_i)=d\beta$ for some $\beta\in Q$. Then
$$
f(x)-\alpha = p(x+\alpha+\sum a_i\otimes sv_i) = d\beta.
$$
On the other hand, $D(x+\alpha+\sum a_i\otimes sv_i)=0$ implies that
$$
dx+\sum (-1)^{|a_i|}a_i\otimes v_i=0, \quad d\alpha+\sum (-1)^{|a_i|}a_i\otimes f(v_i)=0, \quad \sum \left( da_i\otimes v_i-(-1)^{|a_i|}a_i\cdot S(dv_i) \right) =0.
$$

Write $x=\sum b_j\otimes w_j$ with $b_j\in A$ and $w_j\in V$. Then $S(x)=\sum (-1)^{|b_j|} b_j\otimes sw_j$. Hence,
\begin{align*}
D(Sx) &= \sum b_j\otimes w_j+\sum b_j\otimes f(w_j)+\sum \left((-1)^{|b_j|} db_j\otimes sw_j-b_j\cdot S(dw_j)\right) \\
&= x+f\left(\sum b_j\otimes w_j\right)+\sum\left(-S(db_j\otimes w_j)-(-1)^{|b_j|} S(b_j\cdot dw_j)\right) \\
&= x+f(x)-S\left(\sum d(b_j\otimes w_j) \right) \\
&= x+f(x)-S(dx).
\end{align*}
As $dx=-\sum (-1)^{|a_i|}a_i\otimes v_i$, it follows that $S(dx)=-\sum a_i\otimes sv_i$. Thus,
$$
x+\alpha+\sum a_i\otimes sv_i = (\alpha-f(x))+(x+f(x)-S(dx)) = D(-\beta+Sx).
$$
This proves that $p$ is injective on cohomology.

For any cocycle $\alpha\in Q$, the element $-\alpha$ is a cocycle in $\widetilde{Q}$ and satisfies $p(-\alpha)=\alpha$. So $p$ is surjective on cohomology.

(ii) Construct $p:\widetilde{Q}\stackrel{\simeq}{\longrightarrow} Q$ and $F:P\to\widetilde{Q}$ as in (i), so that $p\circ F=f$. Since $g\circ f\sim h$, there exists a graded $A$-module morphism $\theta:P\to M$ of degree $-1$ such that $g\circ f-h=d\theta+\theta d$. Define $G:\widetilde{Q}\to M$ by
$$
G|_P=h, \quad G|_Q=-g, \quad G(sv)=-\theta(v) \ (sv\in sV).
$$
For any $a\in A$ and $v\in V$, the last equation gives
$$
(G\circ S)(a\cdot v) = (-1)^{|a|} G(a\otimes sv) = -(-1)^{|a|} a\cdot \theta(v) = -\theta(a\cdot v).
$$
Hence, $G\circ S=-\theta$ on $P$. Furthermore,
$$
(G\circ D)(sv) = G(v+f(v)-S(dv)) = h(v)-(g\circ f)(v)+\theta(dv)=-(d\circ \theta)(v)=(d\circ G)(sv).
$$
Thus, $G$ is a morphism of DG $A$-modules, and by construction, $G\circ F=h$.

(iii) First apply (i) to construct $F:P\to\widetilde{Q}$. This construction depends only on $f$. We then apply the remaining steps of (ii) separately to the homotopies $q\sim\phi\circ f$ and ${\rm id}_P\sim g\circ f$. This yields two morphisms $\Phi:\widetilde{Q}\to P/(I\cdot P)$ and $G:\widetilde{Q}\to P$ satisfying $q=\Phi\circ F$ and ${\rm id}_P=G\circ F$.

(iv) By Proposition \ref{surjective resolution} there exists a surjective semifree resolution $p':\widetilde{Q}'\stackrel{\simeq}{\longrightarrow}\widetilde{Q}$. Applying Proposition \ref{commutative lift} (ii) to $F:P\to\widetilde{Q}$ yields a lift $F':P\to\widetilde{Q}'$ such that $p'\circ F'=F$. Replacing $\widetilde{Q},p,F,G,\Phi$ by $\widetilde{Q}',p\circ p',F',G\circ p',\Phi\circ p'$ respectively preserves all the required properties. Since $F$ is injective, so is $F'$. As both $p$ and $p'$ are surjective, so is $p\circ p'$.
\end{proof}

\begin{cor}\label{Def2 of mcat}
Let $\Lambda V$ be a minimal Sullivan algebra and $P$ a semifree $\Lambda V$-module. If ${\rm mcat}(P)=m$, then there exist a $\Lambda V$-module $Q$, morphisms $f:P\to Q$, $g:Q\to P$, and a quasi-isomorphism $\phi:Q\stackrel{\simeq}{\longrightarrow} P/(\Lambda^{>m} V\cdot P)$ making diagram \eqref{mcat graph} commutative strictly. That is,
$$
\phi\circ f=q, \quad g\circ f= {\rm id}_{P},
$$
where $q:P\to P/(\Lambda^{>m} V\cdot P)$ is the natural projection.
\end{cor}

\subsection{Module Categories over a general KS complex}

In \cite{FHT}, for a general Sullivan algebra $\Lambda V$, $\text{cat}(\Lambda V)$ can also be defined as in Definition \ref{def of cat}. It can be shown that this definition agrees with the category of its minimal model. Similarly, $e(\Lambda V)$ can be defined as in Definition \ref{def of e}. In this subsection, we prove that the module category and the Toomer invariant of DG modules over non-minimal Sullivan algebras admit analogous definitions. These invariants will be used in the proof of our main theorems.

\begin{defn}
Let $\Lambda V$ be a connected KS complex (not necessarily Sullivan or minimal), and $P$ a $\Lambda V$-semifree module. We define $\text{mcat}_{\Lambda V}(P)$ as the least integer $m\geq 0$ for which there exist a semifree $\Lambda V$-module $Q$, morphisms $f:P\to Q$, $g:Q\to P$, and a quasi-isomorphism $\phi:Q\stackrel{\simeq}{\longrightarrow} P/(\Lambda^{>m} V\cdot P)$ making diagram \eqref{mcat graph} commutative (or merely homotopy-commutative by Proposition \ref{replace homotopy}).

$e_{\Lambda V}(P)$ is defined as the least integer $m\geq 0$ such that the projection $q:P\to P/(\Lambda^{>m} V\cdot P)$ induces an injection on cohomology.

For an arbitrary $\Lambda V$-module $M$, we define $\text{mcat}_{\Lambda V}(M)$ and $e_{\Lambda V}(M)$ via any semifree resolution. These definitions are independent of the resolution chosen.
\end{defn}

\begin{rmk}\label{non-connected remark}
The above definitions can be extended to the case where $\Lambda V$ is not connected. In this case, there exists a non-constant degree-zero cocycle $x\in (\Lambda V)^0$. The cohomology classes of all $(x-\epsilon(x))^{m+1} \in \Lambda^{>m} V)$ are nonzero and consequently $q^*:H^*(\Lambda V) \to H^*(\Lambda V/\Lambda^{>m} V)$ is never injective. Hence, $e_{\Lambda V}(\Lambda V)=\text{mcat}_{\Lambda V}(\Lambda V)=\infty$.
\end{rmk}

When $\Lambda V$ is connected, it has a minimal model $\Lambda Z$. Let $\psi:\Lambda Z\stackrel{\simeq}{\longrightarrow}\Lambda V$ be the quasi-isomorphism. Any $\Lambda V$-semifree module $P$ naturally inherits a $\Lambda Z$-module structure via $z\cdot x:=\psi(z)\cdot x$ for $z\in Z,x\in P$.

\begin{thm}\label{mcat over non-minimal}
Under the hypotheses and notation above, we have ${\rm mcat}_{\Lambda V}(P)={\rm mcat}_{\Lambda Z}(P)$ and $e_{\Lambda V}(P)=e_{\Lambda Z}(P)$.
\end{thm}
\begin{proof}
We shall use the following fact: $\Lambda V$ decomposes as a tensor product
$$
\Lambda V=(\Lambda Z,d)\otimes(\Lambda(U\oplus dU),d),
$$
where $\Lambda Z$ is the minimal model of $\Lambda V$ and $\Lambda(U\oplus dU)$ is contractible, meaning that $d:U\to dU$ is an isomorphism \cite[Theorem 2.2]{Halperin}. When $\Lambda V$ is a Sullivan algebra, a proof is also given in Theorem 14.9 of \cite{FHT}. Without loss of generality, we may assume that $\psi:\Lambda Z\stackrel{\simeq}{\longrightarrow}\Lambda V$ is simply the inclusion.

Observe that $P$ is also semifree over $\Lambda Z$. Indeed, write $P\cong\Lambda Z\otimes\Lambda(U\oplus dU)\otimes W$, where $W$ is the union of an increasing sequence of vector spaces satisfying
$$
dW(k)\subset\Lambda Z\otimes\Lambda(U\oplus dU)\otimes W(k-1).
$$
As $\Lambda(U\oplus dU)$ is a KS complex, we can equip $\Lambda(U\oplus dU)\otimes W$ with a filtration $\widetilde{W}(k)$ satisfying
$$
d\widetilde{W}(k)\subset\Lambda Z\otimes\widetilde{W}(k-1).
$$
Define $\widetilde{W}(-1)=0$ and, inductively, let $\widetilde{W}(k)$ be the subspace of elements whose differential lies in $\Lambda Z\otimes\widetilde{W}(k-1)$. An induction on the filtrations of $\Lambda(U\oplus dU)$ and $W$ shows that
$$
\Lambda(U\oplus dU)\otimes W \subset \bigcup\widetilde{W}(k).
$$

Assume $\text{mcat}_{\Lambda V}(P)=m$. This provides a commutative diagram \eqref{mcat graph} of $\Lambda V$-modules. The same data can also be viewed in the $\Lambda Z$-module category. Under this latter view, the projection $q$ factors as
$$
P \stackrel{q_1}{\longrightarrow} P/(\Lambda^{>m} Z\cdot P) \stackrel{q_2}{\longrightarrow} P/(\Lambda^{>m} V\cdot P).
$$
We will show that $q_2$ is a quasi-isomorphism. Consequently, the quotient $P/(\Lambda^{>m} V\cdot P)$ in \eqref{mcat graph} can be replaced by $P/(\Lambda^{>m} Z\cdot P)$. This proof does not rely on the condition $\text{mcat}_{\Lambda V}(P)=m$. Hence, $q_2$ is a quasi-isomorphism for every $m\geq 0$.

Equip $P$ with the decreasing filtration defined by $P(n)=\Lambda^{\geq n} Z\cdot P$, which is isomorphic to $\Lambda^{\geq n}Z \otimes \Lambda(U\oplus dU)\otimes W$. These filtrations descends to the quotient spaces $P/(\Lambda^{>m} Z\cdot P)$ and $P/(\Lambda^{>m} V\cdot P)$, and $q_2$ preserves these filtrations. The induced filtrations on these quotients are finite, so the associated spectral sequences converge. The $E_0$ pages are given by
$$
(\Lambda^{\leq m} Z,0)\otimes(\Lambda(U\oplus dU)\otimes W,\bar{d})
$$
and
$$
\bigoplus_{0\leq i\leq m} (\Lambda^{m-i} Z,0)\otimes\left(\frac{\Lambda(U\oplus dU)\otimes W}{\Lambda^{>i}(U\oplus dU)\cdot (\Lambda(U\oplus dU)\otimes W)},\bar{d}\right),
$$
where $\bar{d}$ denotes the induced differential in each case.

By the construction of $P$, $(\Lambda(U\oplus dU)\otimes W,\bar{d})$ is a semifree module over $(\Lambda(U\oplus dU)$. Then for each fixed $i$ with $0\leq i\leq m$, its tensor product with $\Lambda^{m-i} Z$ is also semifree over $(\Lambda(U\oplus dU)$. Moreover, we have the natural identification
$$
\Lambda^{m-i} Z \otimes \frac{\Lambda(U\oplus dU)\otimes W}{\Lambda^{>i}(U\oplus dU)\cdot (\Lambda(U\oplus dU)\otimes W)} \cong \frac{\Lambda(U\oplus dU)}{\Lambda^{>i}(U\oplus dU)} \otimes_{\Lambda(U\oplus dU)} (\Lambda^{m-i} Z \otimes \Lambda(U\oplus dU)\otimes W).
$$

Note that the projection $\Lambda(U\oplus dU) \to \frac{\Lambda(U\oplus dU)}{\Lambda^{>i}(U\oplus dU)}$ is a quasi-isomorphism for all $i\geq 0$, as $\Lambda^{>i}(U\oplus dU)$ has trivial cohomology. Indeed, since $\Lambda(U\oplus dU)$ is contractible, every cocycle $x\in\Lambda^{>i}(U\oplus dU)$ satisfies $x-c=dy$ for some $c\in\mathbbm{k}$ and some $y\in\Lambda(U\oplus dU)$. As $d$ preserves the word length, $x$ is the image under $d$ of the part of $y$ of word length greater than $i$. It follows that $x$ is a coboundary in $\Lambda^{>i}(U\oplus dU)$ (and necessarily $c=0$).

According to Proposition 6.7(ii) of \cite{FHT}, the projection $\Lambda(U\oplus dU) \to \frac{\Lambda(U\oplus dU)}{\Lambda^{>i}(U\oplus dU)}$ induces a quasi-isomorphism 
$$
\Lambda(U\oplus dU)\otimes_{\Lambda(U\oplus dU)} (\Lambda^{m-i} Z \otimes \Lambda(U\oplus dU)\otimes W) \stackrel{\simeq}{\longrightarrow} \frac{\Lambda(U\oplus dU)}{\Lambda^{>i}(U\oplus dU)} \otimes_{\Lambda(U\oplus dU)} (\Lambda^{m-i} Z \otimes \Lambda(U\oplus dU)\otimes W).
$$
This yields a quasi-isomorphism on the $E_0$-pages, and hence
isomorphisms on subsequent pages. As noted above, these spectral sequences both convergent. This completes the proof that $q_2$ is a quasi-isomorphism.

We now return to the main proof. Recall the commutative diagram \eqref{mcat graph}. Lift $\phi:Q \stackrel{\simeq}{\longrightarrow} P/(\Lambda^{>m} V\cdot P)$ through the surjective quasi-isomorphism $q_2$. This gives a quasi-isomorphism $\Phi: Q \stackrel{\simeq}{\longrightarrow} P/(\Lambda^{>m} Z\cdot P)$ satisfying $q_2\circ\Phi=\phi$. Then
$$
q_2\circ q_1=q=\phi\circ f=q_2\circ\Phi\circ f.
$$
Therefore, Proposition \ref{commutative lift} (i) yields $q_1\sim\Phi\circ f$. By Proposition \ref{replace homotopy}, we can replace this homotopy by an equality, making the following diagram commute strictly:
$$
\xymatrix{
P \ar[r]^f \ar[rd]_{q_1} & Q \ar[rd]^(0.4){\phi}_(0.4){\simeq} \ar[d]_{\Phi}^{\simeq} \\
& P/(\Lambda^{>m} Z\cdot P) \ar[r]^{q_2}_{\simeq} & P/(\Lambda^{>m} V\cdot P)
}
$$
Consequently, the data $P/(\Lambda^{>m} V\cdot P)$, $q$, $\phi$ can be replaced by $P/(\Lambda^{>m} Z\cdot P)$, $q_1$, $\Phi$, respectively. It follows that
$$
\text{mcat}_{\Lambda Z}(P) \leq m = \text{mcat}_{\Lambda V}(P).
$$

Conversely, suppose $\text{mcat}_{\Lambda Z}(P)=m$. We have a diagram of the form \eqref{mcat graph} in the $\Lambda Z$-modules category, with $P/(\Lambda^{>m} V\cdot P)$ replaced by $P/(\Lambda^{>m} Z\cdot P)$. Apply the functor $\Lambda V\otimes_{\Lambda Z}-$ to all the modules and morphisms. Then we obtain the following commutative diagram of $\Lambda V$-modules:
$$
\xymatrix{
\Lambda V\otimes_{\Lambda Z}P \ar[rd] \ar[rdd] \ar[rr]^{id} & & \Lambda V\otimes_{\Lambda Z}P\\
& \Lambda V\otimes_{\Lambda Z}Q \ar[ur] \ar[d]_{\simeq} \\
& \Lambda V\otimes_{\Lambda Z}\left(P/(\Lambda^{>m} Z\cdot P)\right)
}
$$
The vertical arrow is a quasi-isomorphism by Proposition 6.7(ii) of \cite{FHT}, since $\Lambda V$ is a semifree $\Lambda Z$-module.

Observe that $\Lambda V\otimes_{\Lambda Z}P \cong (\Lambda(U\oplus dU),d)\otimes(P,d)$ is a $\Lambda V$-semifree module quasi-isomorphic to $P$. Moreover,
$$
\Lambda V\otimes_{\Lambda Z}\left(P/(\Lambda^{>m} Z\cdot P)\right) \cong (\Lambda(U\oplus dU)\otimes P)/(\Lambda^{>m} Z\cdot(\Lambda(U\oplus dU)\otimes P))
$$
By the same argument used for $q_2$, we have a quasi-isomorphism
$$
(\Lambda(U\oplus dU)\otimes P)/(\Lambda^{>m} Z\cdot(\Lambda(U\oplus dU)\otimes P)) \stackrel{\simeq}{\longrightarrow} (\Lambda(U\oplus dU)\otimes P)/(\Lambda^{>m} V\cdot(\Lambda(U\oplus dU)\otimes P)).
$$
Therefore,
$$
\text{mcat}_{\Lambda V}(P) = \text{mcat}_{\Lambda V}(\Lambda(U\oplus dU)\otimes P) \leq m = \text{mcat}_{\Lambda Z}(P).
$$
This proves that $\text{mcat}_{\Lambda V}(P) = \text{mcat}_{\Lambda Z}(P)$.

Finally, as established above, $q_2:P/(\Lambda^{>m} Z\cdot P) \stackrel{\simeq}{\longrightarrow} P/(\Lambda^{>m} V\cdot P)$ is a quasi-isomorphism for every $m\geq 0$. As a consequence, $q_1:P\to P/(\Lambda^{>m} Z\cdot P)$ is injective in cohomology if and only if $q_2\circ q_1$ is injective in cohomology. Thus, $e_{\Lambda V}(P) = e_{\Lambda Z}(P)$.
\end{proof}

\begin{cor}
For any connected Sullivan algebra $\Lambda V$, we have ${\rm mcat}_{\Lambda V}(\Lambda V)={\rm cat}(\Lambda V)$.
\end{cor}
\begin{proof}
Let $\Lambda Z$ be the minimal Sullivan model of $\Lambda V$, then
$$
\text{mcat}_{\Lambda V}(\Lambda V) = \text{mcat}_{\Lambda Z}(\Lambda V) = \text{mcat}_{\Lambda Z}(\Lambda Z) = \text{cat}(\Lambda Z) = \text{cat}(\Lambda V).
$$
\end{proof}

\section{Main Theorems}

\subsection{Estimating the module category}

We begin with a lemma that will be used in the proof of the main theorem of this subsection, as well as that of the next.

\begin{lem}\label{ideal factorization}
Let $\Lambda W$ be a Sullivan algebra with ${\rm cat}(\Lambda W)=n<\infty$. For every $k\geq 0$, there exist surjective semifree resolutions
$$
\phi:P\to\Lambda^{>k} W, \quad \psi:Q\to\Lambda^{>k} W/\Lambda^{>k+n} W
$$
of $\Lambda W$-modules, together with $\Lambda W$-module morphisms
$$
P\stackrel{f}{\longrightarrow} Q\stackrel{g}{\longrightarrow} P
$$
such that $f$ is a lift of the natural projection ${\rm pr}:\Lambda^{>k} W \to \Lambda^{>k} W/\Lambda^{>k+n} W$, and
$$
g\circ f={\rm id}_P, \quad \psi\circ f = {\rm pr}\circ\phi.
$$
Equivalently, the following diagram commutes:
$$
\xymatrix{
P \ar[rd]^f \ar[dd]^{\phi}_{\simeq} \ar[rr]^{{\rm id}} & & P\\
& Q \ar[ur]^g \ar[d]^{\psi}_{\simeq} \\
\Lambda^{>k} W \ar[r]^(0.35){{\rm pr}} & \Lambda^{>k} W/\Lambda^{>k+n} W
}
$$
\end{lem}
\begin{proof}

In the first two steps, we treat the case where $\Lambda W$ is minimal. If $n=0$, the injectivity of $H(\Lambda W)\to H(\Lambda W/\Lambda^{>0} W) = \mathbbm{k}$ forces $W=0$, so the lemma holds trivially. We may therefore assume $n\geq 1$ in the minimal case.

We shall use the following proposition, which collects the facts we need from Proposition 9.4, its proof and Corollary 9.1 of \cite{FHT2}.

\begin{prop}\label{to prove Hess' thm}
For each $m\geq 1$, there exists a $\Lambda W$-semifree resolution
$$
\chi_m:\Lambda W\otimes(\mathbbm{k}\oplus Z(m)) \stackrel{\simeq}{\longrightarrow} \Lambda W/\Lambda^{>m} W
$$
satisfying the following properties:

(i) The restriction of $\chi_m$ to $\Lambda W\otimes\mathbbm{k}$ coincides with the natural projection $\Lambda W \to \Lambda W/\Lambda^{>m} W$.

(ii) $\chi_m(Z(m))=0$.

(iii) $dZ(m) \subset \Lambda^{>m}W\otimes\mathbbm{k} \bigoplus \Lambda^+W\otimes Z(m)$.

(iv) $\Lambda W\otimes(\mathbbm{k}\oplus Z(m))$ is a submodule of a relative Sullivan algebra.

(v) $\Lambda W\otimes(\mathbbm{k}\oplus Z(m))$ admits a bidegree such that
$$
(\Lambda^p W)^{p+q} = (\Lambda W)^{p,q}, \quad Z(m)^i = Z(m)^{m,i-m}.
$$
The bidegree on $\Lambda W/\Lambda^{>m} W$ is induced accordingly, where the first index records word length. Let $d_1$ denote the $(1,0)$-component of $d$. Then $\chi_m$ also induces a quasi-isomorphism
$$
(\Lambda W\otimes(\mathbbm{k}\oplus Z(m)), d_1) \stackrel{\simeq}{\longrightarrow} (\Lambda W/\Lambda^{>m} W, d_1).
$$
\end{prop}

\noindent\textbf{Step 1.} For each $m\geq 1$, construct a $\Lambda W$-semifree resolution of $\Lambda^{>m} W$.

Construct the semifree resolution $\chi_m$ as in Proposition \ref{to prove Hess' thm}. Let
$$
N_m:=\Lambda W\otimes(\mathbbm{k}\oplus Z(m)), \quad M_m:=\frac{\Lambda W\otimes(\mathbbm{k}\oplus Z(m))}{\Lambda W\otimes\mathbbm{k}}.
$$
The quotient $M_m$ is again semifree. Note that, as graded vector spaces, $N_m=\Lambda W\oplus M_m$.

Define a degree-one map $\theta_m: M_m\to \Lambda W$ by requiring that, for $y\in M_m$ and $x\in\Lambda W$,
$$
d(x,y) = (d_{\Lambda W}x+\theta_m y, d_{M_m} y).
$$
Since $\chi_m d(M_m)=d\chi_m(M_m)=0$, we have $\theta_m(M_m) \subset \ker(\Lambda W \to \Lambda W/\Lambda^{>m} W)=\Lambda^{>m} W$.

Let $s^{-1}M_m$ denote the desuspension of $M_m$, with induced differential and $\Lambda W$-action given by the Koszul sign rule:
$$
(s^{-1}M_m)^i = M_m^{i-1}, \quad d_{s^{-1}M_m}(s^{-1}y)=-s^{-1}(d_{M_m} y), \quad x\cdot(s^{-1}y)=(-1)^{|x|}s^{-1}(x\cdot y) 
$$
for $y\in M_m$ and $x\in\Lambda W$.

Define
$$
\eta_m:s^{-1}M_m\to \Lambda^{>m} W, \quad s^{-1}y\mapsto \theta_m y.
$$
Since $d^2=0$ on $N_m$, we have
$$
d_{\Lambda W}\circ\theta_m+\theta_m\circ d_{M_m}=0.
$$
Equivalently,
$$
d_{\Lambda W}\circ\eta_m-\eta_m\circ d_{s^{-1}M_m}=0.
$$
This proves that $\eta_m$ is a chain map.

To verify the $\Lambda W$-linearity, take $x\in\Lambda W$ and $y\in M_m$. The differential in $N_m$ is given by the graded Leibniz rule:
$$
d(x\cdot y)=(dx)\cdot y+(-1)^{|x|}x\cdot dy.
$$
Projecting onto the $\Lambda W$-component yields
$$
\theta_m(x\cdot y) = (-1)^{|x|}x\cdot \theta_m y.
$$
Consequently,
$$
\eta_m(x\cdot s^{-1}y) = (-1)^{|x|}\eta_m(s^{-1}(x\cdot y)) = x\cdot \eta_m(s^{-1}y),
$$
Thus, $\eta_m$ is a morphism of $\Lambda W$-modules.

Since $\chi_m$ is a surjective quasi-isomorphism, its kernel is acyclic. Now consider the exact sequence
$$
0\to \Lambda^{>m} W \to \ker\chi_m \to M_m \to 0.
$$
The acyclicity of $\ker\chi_m$ implies that the connecting map $\partial:H^{i-1}(M_m) \to H^i(\Lambda^{>m} W)$ is an isomorphism. Moreover, by construction, $\partial$ is the morphism induced by $\theta_m$ on cohomology. Therefore, $\eta_m$ induces the following composite of isomorphisms on cohomology
$$
H^i(s^{-1}M_m) \xrightarrow[\cong]{} H^{i-1}(M_m) \xrightarrow[\cong]{\partial} H^i(\Lambda^{>m} W).
$$
Consequently, $\eta_m$ is a semifree resolution of $\Lambda^{>m} W$.

\ 

\noindent\textbf{Step 2.} Construct a null-homotopy for the composite $s^{-1}M_{k+n} \xrightarrow{\eta_{k+n}} \Lambda^{>k+n} W \stackrel{\iota}{\longrightarrow} \Lambda^{>k} W$.

Following the proof of Theorem 29.14 of \cite{FHT}, we first construct a morphism $\xi:N_{k+n}\to N_n$ such that diagram
$$
\xymatrix
{
N_{k+n} \ar[r]^(0.35){\chi_{k+n}}_(0.35){\simeq} \ar[d]_{\xi} & \Lambda W/\Lambda^{>k+n} W \ar[d]^{\text{pr}} \\
N_n \ar[r]^(0.38){\chi_n}_(0.38){\simeq} & \Lambda W/\Lambda^{>n} W
}
$$
commutes, and such that $\xi|_{\Lambda W}=\text{id}_{\Lambda W}$. Moreover, using the bidegree from Proposition \ref{to prove Hess' thm}(v), we may arrange that $\xi$ respects the filtration in the sense that
$$
\xi(N_{k+n}^{p,*}) \subset N_n^{\geq p, *}.
$$

When $k=0$, we can take $\xi=\text{id}_{N_n}$ trivially, so we assume $k\geq 1$ in the construction below.

By Proposition \ref{to prove Hess' thm}(iv), $Z(k+n)$ admits a filtration
$$
0=Z(k+n)(-1) \subset Z(k+n)(0) \subset Z(k+n)(1) \subset \ldots \subset Z(k+n)(j) \subset \ldots
$$
such that
$$
Z(k+n)= \bigcup_j Z(k+n)(j), \quad dZ(k+n)(j)\subset \Lambda^{>k+n} W \oplus \Lambda^+ W\otimes Z(k+n)(j-1).
$$

It suffices to define $\xi$ inductively on $Z(k+n)(j)$. Assume $\xi$ has been defined on $Z(k+n)(j-1)$. Choose a homogeneous basis of a complement of $Z(k+n)(j-1)$ in $Z(k+n)(j)$. For each basis element $z$, we have
$$
dz\in \Lambda^{>k+n}W \oplus \Lambda^+W\otimes Z(k+n)(j-1).
$$
Thus, $\xi(dz)$ is defined, and by hypothesis it is in $N_n^{>k+n,*}$.

Suppose $\xi(dz)$ has total degree $i$. By Proposition \ref{to prove Hess' thm}(iv), both $W$ and $Z(n)$ are positively graded. So $\xi(dz)\in N_n^{\leq i+n-1,*}$. It can therefore be decomposed as
$$
\xi(dz)=\sum_{p=k+n+1}^{i+n-1} \alpha_p, \quad \alpha_p \in N_n^{p,i-p}.
$$

By the construction of the differential on $N_n$, it increases the first index by at least 1. Since $\xi(dz)$ is a cocycle, we must have $d_1 \alpha_{k+n+1}=0$. As $\chi_n$ preserves the bidegree, Proposition \ref{to prove Hess' thm}(v) implies that $H(N_n, d_1)$ is concentrated in bidegrees $(\leq n, *)$. Thus, $\alpha_{k+n+1}=d_1\beta_{k+n}$ for some $\beta_{k+n}\in N_n^{k+n, i-1-k-n}$. Then
$$
\xi(dz)-d\beta_{k+n} \in N_n^{>k+n+1,*} \cap N_n^{\leq i+n-1,*},
$$
where the first index is bounded above by $i+n-1$ because $\xi(dz)$ has total degree $i$. Iterating this procedure, we eventually obtain $\xi(dz)=d\beta$ for some $\beta\in N_n^{\geq k+n,*}$.

Define $\xi(z)=\beta$. This closes the induction. Since $\chi_{k+n},\chi_n$ and pr preserve the bidegree, and $\xi$ does not decrease the first index, it follows that
$$
\chi_n\circ\xi(Z(k+n)) = 0 = \text{pr}\circ\chi_{k+n}(Z(k+n)).
$$
Thus, the required diagram commutes.

On the other hand, since $\text{cat}(\Lambda W)=n$, there exists a relative Sullivan algebra $\Lambda W\otimes\Lambda V$ and morphisms making the following diagram commute:
$$
\xymatrix
{
\Lambda W \ar[rd] \ar[rdd]_{\text{pr}} \ar[rr]^{\text{id}} & & \Lambda W\\
& \Lambda W\otimes\Lambda V \ar[ur] \ar[d]_{\simeq} \\
 & \Lambda W/\Lambda^{>n} W
}
$$

By Proposition \ref{commutative lift}, there exists a $\Lambda W$-module morphism $N_n\to\Lambda W\otimes\Lambda V$ such that its composite with the projection $\Lambda W\otimes\Lambda V \to \Lambda W/\Lambda^{>n} W$ is $\chi_n$. This morphism is a quasi-isomorphism. For degree reasons, it sends $1\in\Lambda W\subset N_n$ to the identity element in $\Lambda W\otimes\Lambda V$. Thus, its composite with the map $\Lambda W\otimes\Lambda V \to \Lambda W$ in the above diagram restricts to the identity on $\Lambda W$.

By degree considerations, the above composite $N_n\to \Lambda W$ sends $Z(n)$ to $\Lambda^+ W$. Since
$$
N_n^{p,*}=\Lambda^p W\oplus\Lambda^{p-n} W\otimes Z(n),
$$
it maps $N_n^{p,*}$ into $\Lambda^{>p-n} W$. Let $\zeta$ denote the composite
$$
N_{k+n} \stackrel{\xi}{\longrightarrow} N_n \to \Lambda W.
$$
As $\xi$ does not decrease the first index of the bidegree, we have $\zeta(Z(k+n))\subset\Lambda^{>k} W$.

Therefore, we can construct a $\Lambda W$-linear map
$$
h:s^{-1}M_{k+n} \to \Lambda^{>k} W, \quad s^{-1}y\mapsto\zeta(0,y)
$$
of degree $-1$. Since $\zeta$ commutes with the differential and restricts to the identity on $\Lambda W$, we have
\begin{align*}
(dh+hd)(s^{-1}y) &= d\zeta(0,y)-hs^{-1}(d_{M_{k+n}}y) \\
&= \zeta d(0,y)-\zeta(0,d_{M_{k+n}}y) \\
&= \zeta(\theta_{k+n}y,0) \\
&= \theta_{k+n}y \\
&= \iota\eta_{k+n}(s^{-1}y),
\end{align*}
where $\iota:\Lambda^{>k+n} W \to \Lambda^{>k} W$ is the inclusion. This proves that $\iota\circ\eta_{k+n}$ is null-homotopic.

\ 

\noindent\textbf{Step 3.} Consider the remaining case where $\Lambda W$ is not necessarily minimal and $n\geq 0$.

For a general $\Lambda W$, we can write
$$
\Lambda W \cong (\Lambda Z, d)\otimes(\Lambda(U\oplus dU), d),
$$
where $\Lambda Z$ is minimal and $d:U\to dU$ is an isomorphism \cite[Theorem 14.9]{FHT}.

As shown in the proof of Theorem \ref{mcat over non-minimal}, for each $m\geq 0$, the projection
$$
q_2: \Lambda W/(\Lambda^{>m} Z\otimes\Lambda(U\oplus dU)) \to \Lambda W/\Lambda^{>m} W
$$
is a quasi-isomorphism (with $P=\Lambda W$). Consequently, the inclusion
$$
\iota_m:\Lambda^{>m} Z\otimes\Lambda(U\oplus dU) \to \Lambda^{>m} W
$$
is also a quasi-isomorphism.

When $n\geq 1$, Step 1 provides a $\Lambda Z$-semifree resolution
$$
\overline{\eta}_{m}: s^{-1}\overline{M}_m \stackrel{\simeq}{\longrightarrow} \Lambda^{>m} Z
$$
for $m\geq 1$, and Step 2 provides a null-homotopy $\overline{h}$ for the composite
$$
s^{-1}\overline{M}_{k+n} \xrightarrow{\overline{\eta}_{k+n}} \Lambda^{>k+n} Z \stackrel{\overline{\iota}}{\longrightarrow} \Lambda^{>k} Z.
$$
When $n=0$, as established earlier, $Z=0$. We then take $s^{-1}\overline{M}_m=0$, $\overline{\eta}_{m}=0$, and $\overline{h}=0$ for all $m\geq 0$.

Define
$$
s^{-1}M_m = (s^{-1}\overline{M}_m,d)\otimes(\Lambda(U\oplus dU), d).
$$
Let $\eta_m$ be the composite
$$
s^{-1}M_m \xrightarrow{\overline{\eta}_{m}\otimes\text{id}} \Lambda^{>m} Z\otimes\Lambda(U\oplus dU) \xrightarrow[\simeq]{\iota_m} \Lambda^{>m} W,
$$
and let $h$ be the composite
$$
s^{-1}M_{k+n} \xrightarrow{\overline{h}\otimes\text{id}} \Lambda^{>k} Z\otimes\Lambda(U\oplus dU) \xrightarrow[\simeq]{\iota_k} \Lambda^{>k} W.
$$
Then $\eta_m$ is a $\Lambda W$-semifree resolution, and $h$ is $\Lambda W$-linear.

Let $d_M,d_{\overline{M}},d_W,d_Z,d_U$ denote the differentials on $s^{-1}M_m,s^{-1}\overline{M}_m,\Lambda W,\Lambda Z,\Lambda(U\oplus dU)$, respectively. Let $\iota: \Lambda^{>k+n} W\to \Lambda^{>k} W$ denote the inclusion. Then
\begin{align*}
d_W h+hd_M &= \iota_k(d_Z\otimes\text{id}+\text{id}\otimes d_U)(\overline{h}\otimes\text{id}) + \iota_k(\overline{h}\otimes\text{id})(d_{\overline{M}}\otimes\text{id}+\text{id}\otimes d_U) \\
&= \iota_k(d_Z\overline{h}\otimes\text{id})-\iota_k(\overline{h}\otimes d_U)+\iota_k(\overline{h}d_{\overline{M}}\otimes\text{id})+\iota_k(\overline{h}\otimes d_U) \\
&= \iota_k((d_Z\overline{h}+\overline{h}d_{\overline{M}})\otimes\text{id}) \\
&= \iota_k((\overline{\iota}\circ\overline{\eta}_{k+n})\otimes\text{id}) \\
&= \iota\iota_{k+n}(\overline{\eta}_{k+n}\otimes\text{id}) \\
&= \iota\eta_{k+n}.
\end{align*}
Thus, $h$ is a null-homotopy for $\iota\circ\eta_{k+n}$.

\ 

\noindent\textbf{Step 4.} Lift $\iota\circ\eta_{k+n}$ and obtain a null-homotopy for the lift.

Choose a surjective semifree resolution $\phi:P\stackrel{\simeq}{\longrightarrow} \Lambda^{>k} W$. By Proposition \ref{commutative lift}, $\iota\circ\eta_{k+n}$ lifts to a morphism $\eta':s^{-1}M_{k+n}\to P$ making the following diagram commute:
$$
\xymatrix
{
& P \ar[d]_{\phi}^{\simeq} \\
s^{-1}M_{k+n} \ar@{-->}[ur]^(0.55){\eta'} \ar[r]_(0.55){\iota\circ\eta_{k+n}} & \Lambda^{>k} W
}
$$
Moreover, the homotopy class of $\eta'$ is uniquely determined. Since we have proved that $\iota\circ\eta_{k+n}$ is null-homotopic, $\eta'$ is null-homotopic as well. Thus, there exists a $\Lambda W$-linear map $h':s^{-1}M_{k+n}\to P$ of degree $-1$ such that $\eta'=dh'+h'd$.

\ 

\noindent\textbf{Step 5.} Construct a semifree resolution $\psi:Q\to \Lambda^{>k} W/\Lambda^{>k+n} W$.

Define $Q$ to be the mapping cone Cone$(\eta')=(s^{-1}M_{k+n})[1]\oplus P$ of $\eta'$, where
\begin{align*}
& (V[1])^i=V^{i+1} \text{ for any graded vector space } V, \\
& d(s^{-1}y,z)=(-d_{s^{-1}M_{k+n}}(s^{-1}y),-\eta'(s^{-1}y)+d_P(z)), \\
& x\cdot(s^{-1}y,z) = ((-1)^{|x|}x\cdot s^{-1}y, x\cdot z) = (s^{-1}(x\cdot y), x\cdot z)
\end{align*}
for $x\in\Lambda W$, $s^{-1}y\in s^{-1}M_{k+n}$ and $z\in P$.
Since $s^{-1}M_{k+n}$ and $P$ are semifree, $Q$ is semifree as well by Lemma 6.3 of \cite{FHT}.

Define
$$
\psi: Q\to \Lambda^{>k} W/\Lambda^{>k+n} W, \quad (s^{-1}y,z)\mapsto \text{pr}(\phi z),
$$
where $\text{pr}: \Lambda^{>k} W \to \Lambda^{>k} W/\Lambda^{>k+n} W$ is the natural projection. $\psi$ is surjective by construction.

A direct computation gives
$$
d\psi(s^{-1}y,z)=(d\circ\text{pr}\circ\phi)z=\text{pr}(\phi dz).
$$
On the other hand,
$$
\psi d(s^{-1}y,z)=-(\text{pr}\circ\phi\circ\eta')(s^{-1}y)+\text{pr}(\phi dz).
$$
By construction, $\phi\circ\eta'=\iota\circ\eta_{k+n}$. Moreover, $\im\eta_{k+n} \subset \Lambda^{>k+n} W = \ker\text{pr}$. Hence,
$$
\text{pr}\circ\phi\circ\eta'=\text{pr}\circ\iota\circ\eta_{k+n}=0.
$$
This proves that $\psi$ is a cochain map.

The map $\psi$ is $\Lambda W$-linear, because $\phi$ and pr are $\Lambda W$-module morphisms. It remains to show that $\psi$ is a quasi-isomorphism.

Define
$$
\Psi:Q\to\text{Cone}(\iota), \quad (s^{-1}y,z)\mapsto(\eta_{k+n}(s^{-1}y),\phi(z)).
$$
A direct computation shows that $\Psi$ is a cochain map:
\begin{align*}
d\Psi(s^{-1}y,z) &= (-d\eta_{k+n}(s^{-1}y),-(\iota\circ\eta_{k+n})(s^{-1}y)+d\phi(z)) \\
&= (\eta_{k+n}(-ds^{-1}y),-(\phi\circ\eta')(s^{-1}y)+\phi(dz)) \\
&= \Psi(-ds^{-1}y,-\eta'(s^{-1}y)+dz) \\
&= \Psi d(s^{-1}y,z).
\end{align*}
Thus, $\Psi$ induces the following commutative diagram:
$$
\xymatrix{
0 \ar[r] & P \ar@{^{(}->}[r] \ar[d]^{\simeq}_{\phi} & Q \ar[r]^(0.38){-\text{pr}} \ar[d]_{\Psi} & s^{-1}M_{k+n}[1] \ar[r] \ar[d]^{\simeq}_{\eta_{k+n}[1]} & 0 \\
0 \ar[r] & \Lambda^{>k} W \ar@{^{(}->}[r] & \text{Cone}(\iota) \ar[r]^(0.42){-\text{pr}} & \Lambda^{>k+n} W[1] \ar[r] & 0
}
$$
The horizontal rows are exact, and the vertical maps $\phi$ and $\eta_{k+n}[1]$ are quasi-isomorphisms. By the five lemma, $\Psi$ is a quasi-isomorphism as well.

On the other hand, $\Lambda^{>k+n} W \oplus \Lambda^{>k+n} W[1]$ is a subcomplex of Cone$(\iota)$, and is in fact the mapping cone of the identity map on $\Lambda^{>k+n} W$. Thus, it is acyclic by \cite[Corollary 1.5.4]{weibel}. It follows that the projection
$$
\text{Cone}(\iota) \to \frac{\text{Cone}(\iota)}{\Lambda^{>k+n} W \oplus \Lambda^{>k+n} W[1]} = \frac{\Lambda^{>k} W}{\Lambda^{>k+n} W}
$$
is a quasi-isomorphism.

Finally, the composite of $\Psi$ with the above projection is precisely $\psi$. Therefore, $\psi$ is a quasi-isomorphism, and hence a semifree resolution.

\ 

\noindent\textbf{Step 6.} Construct $f:P\to Q$ and $g:Q\to P$.

Define
$$
f: P\to Q, \quad z\mapsto (0,z),
$$
and
$$
g: Q\to P, \quad (s^{-1}y,z)\mapsto -h'(s^{-1}y)+z.
$$
The map $f$ is clearly a $\Lambda W$-module morphism. For $g$, we compute
\begin{align*}
dg(s^{-1}y,z) &= -dh'(s^{-1}y)+dz \\
&= -\eta'(s^{-1}y)+h'd(s^{-1}y)+dz \\
&= g(-d(s^{-1}y), -\eta'(s^{-1}y)+dz) \\
&= gd(s^{-1}y,z).
\end{align*}
This proves that $g$ is a cochain map. Since $h'$ is $\Lambda W$-linear, so is $g$.

By construction, $g\circ f=\text{id}_P$. Moreover, for $z\in P$,
$$
(\psi\circ f)(z) = \psi(0,z) = (\text{pr}\circ\phi)(z).
$$
This completes the proof.
\end{proof}

We now apply Lemma \ref{ideal factorization} to prove the main theorem of this subsection.

\begin{thm}\label{main}
Let $\Lambda V=\Lambda Z\otimes\Lambda W$ be a relative Sullivan algebra. Suppose that $P$ is a $\Lambda V$-semifree module of the form $\Lambda V\otimes U$, where $U=\bigcup_k U(k)$ is an increasing union of subspaces with
$$
dU(k) \subset \Lambda Z\otimes U(k-1).
$$
Let $\overline{P}=\Lambda Z\otimes U$, which is a $\Lambda Z$-semifree module. If ${\rm mcat}_{\Lambda Z}(\overline{P})=m$ and ${\rm cat}(\Lambda W)=n$,
then
$$
{\rm mcat}_{\Lambda V}(P)\leq (m+1)(n+1)-1.
$$
\end{thm}
\begin{proof}
By hypothesis, there exists a $\Lambda Z$-semifree module $\overline{Q}$ equipped with a quasi-isomorphism $\overline{Q} \stackrel{\simeq}{\longrightarrow} \overline{P}/(\Lambda^{>m} Z\cdot\overline{P})$ making the following diagram of $\Lambda Z$-modules commute:
$$
\xymatrix{
\overline{P} \ar[rd] \ar[rdd]_{\text{pr}} \ar[rr]^{\text{id}} & & \overline{P}\\
& \overline{Q} \ar[ur] \ar[d]_{\simeq} \\
& \overline{P}/(\Lambda^{>m} Z\cdot\overline{P})
}
$$
All natural projections are denoted by pr in this proof. Apply $\Lambda V\otimes_{\Lambda Z}-$ to the modules and morphisms in the above diagram. Since $\Lambda V$ is $\Lambda Z$-semifree \cite[Lemma 14.1]{FHT}, this functor preserves quasi-isomorphisms \cite[Proposition 6.7]{FHT}. Observe that $P=\Lambda V\otimes_{\Lambda Z}\overline{P}$, and
$$
\Lambda V\otimes_{\Lambda Z}(\overline{P}/(\Lambda^{>m} Z\cdot\overline{P})) = P/(\Lambda^{>m} Z\cdot P).
$$
Set $P_{m+1}=\Lambda V\otimes_{\Lambda Z}\overline{Q}$, which is a $\Lambda V$-semifree module. We thus obtain the following commutative diagram of $\Lambda V$-modules:
$$
\xymatrix{
P \ar[rd] \ar[rdd]_{\text{pr}} \ar[rr]^{\text{id}} & & P\\
& P_{m+1} \ar[ur] \ar[d]_{\simeq} \\
& P/(\Lambda^{>m} Z\cdot P)
}
$$

Let
$$
P^{\geq p,\geq q}:=(\Lambda^p Z\otimes\Lambda^q W)\cdot P=\Lambda^{\geq p}Z\otimes\Lambda^{\geq q}W\otimes U
$$
denote the subspace of $P$. Note that $P^{\geq p,\geq q}$ is not closed under differential unless $q=0$. Hence, it is not a DG submodule in general.

On the other hand, both $P^{\geq p,\geq q}+P^{\geq p+1,\geq 0}$ and $P^{\geq p+1,\geq 0}$ are DG submodules of $P$. We therefore define the quotient $\Lambda V$-module
$$
P^{p,\geq q} := \frac{P^{\geq p,\geq q}+P^{\geq p+1,\geq 0}}{P^{\geq p+1,\geq 0}}.
$$
This is a DG module over $(\Lambda W, \bar{d})$, because the $\Lambda^+ Z$ action on it is trivial. As a graded vector space, it is isomorphic to $\Lambda^p Z\otimes\Lambda^{\geq q} W\otimes U$.

We use $P^{p,[q,r]}$ to denote the quotient $\Lambda W$-module $P^{p,\geq q}/P^{p,\geq r+1}$. As a graded vector space, it is isomorphic to $\bigoplus_{q\leq j\leq r} \Lambda^p Z\otimes\Lambda^j W\otimes U$.

Then set $I_{m+1}:=P^{\geq m+1,\geq 0}$, and define inductively, for $0\leq k\leq m$,
$$
I_k:=I_{k+1}+P^{\geq k,\geq (m+1-k)(n+1)}.
$$
As 
$$
dP^{\geq p,\geq q}\subset P^{\geq p+1,\geq q-1}+P^{\geq p,\geq q},
$$
all $I_k$ are closed under the differential. Thus, they are a DG submodules.

Now suppose that for some $k<m+1$, we have constructed $\Lambda V$-semifree modules $P_{k+1},P_{k+2},\ldots,P_{m+1}$ such that for $k<i\leq m+1$, there are surjective quasi-isomorphisms $P_i \stackrel{\simeq}{\longrightarrow} P/I_i$, and morphisms $P_{i+1} \to P_i$, $P_i \to P_{i+1}$ making the following diagram commute:
$$
\xymatrix{
P_{i+1} \ar[rd] \ar[rr]^{\text{id}} \ar[dd]_{\simeq} & & P_{i+1} \\
& P_i \ar[ur] \ar[d]_{\simeq} \\
P/I_{i+1} \ar[r]^{\text{pr}} & P/I_i
}
$$
Here we set $P_{m+2}=P$ and $I_{m+2}=0$. We will construct $P_k$ and the corresponding morphisms.

First observe that any surjective quasi-isomorphism $\widetilde{P}_{k+1}\to P_{k+1}$ admits a lift $P_{k+2}\to \widetilde{P}_{k+1}$ that makes the following diagram commute:
$$
\xymatrix{
P_{k+2} \ar@{-->}[rd] \ar[rdd] \ar[rr]^{\text{id}} \ar[ddd]_{\simeq} & & P_{k+2} \\
& \widetilde{P}_{k+1} \ar[d]_{\simeq} \\
& P_{k+1} \ar[uur] \ar[d]_{\simeq} \\
P/I_{k+2} \ar[r]^{\text{pr}} & P/I_{k+1}
}
$$
Thus, we can replace $P_{k+1}$ by $\widetilde{P}_{k+1}$.

Next consider the subspace
$$
(P^{\geq k,\geq(m-k)(n+1)+1}+I_{k+1})/I_{k+1} \subset P/I_{k+1}.
$$
Since
$$
dP^{\geq k,\geq(m-k)(n+1)+1} \subset P^{\geq k,\geq(m-k)(n+1)+1}+P^{\geq k+1,\geq(m-k)(n+1)} \subset P^{\geq k,\geq(m-k)(n+1)+1}+I_{k+1}
$$
and $\Lambda^+ Z\cdot P^{\geq k,\geq(m-k)(n+1)+1} \subset I_{k+1}$, this subspace is a DG submodule over $\Lambda V$ on which the $\Lambda^+ Z$-action is trivial. Hence, it is also a DG module over $(\Lambda W, \bar{d})$, and is isomorphic to
$$
(\Lambda^k Z,d_0)\otimes(\Lambda^{\geq(m-k)(n+1)+1} W,\bar{d})\otimes (U,\bar{d}) \cong P^{k,\geq(m-k)(n+1)+1}.
$$
Here $d_0$ denotes the linear part of the differential on $\Lambda Z$, and $\bar{d}$ denotes the induced differentials on the quotients.

For each $k\leq m$, Lemma \ref{ideal factorization} provides surjective $\Lambda W$-semifree resolutions
$$
M'_k \stackrel{\simeq}{\longrightarrow} \Lambda^{>(m-k)(n+1)} W, \quad Q'_k \stackrel{\simeq}{\longrightarrow} \Lambda^{>(m-k)(n+1)} W/\Lambda^{>(m-k)(n+1)+n} W
$$
fitting into the commutative diagram:
$$
\xymatrix{
M'_k \ar[rd] \ar[dd]_{\simeq} \ar[rr]^{\text{id}} & & M'_k\\
& Q'_k \ar[ur] \ar[d]_{\simeq} \\
\Lambda^{>(m-k)(n+1)} W \ar[r]^(0.35){\text{pr}} & \Lambda^{>(m-k)(n+1)} W/\Lambda^{>(m-k)(n+1)+n} W
}
$$

Define
$$
M_k:=\Lambda^k Z\otimes M'_k\otimes U, \quad Q_k:=\Lambda^k Z\otimes Q'_k\otimes U.
$$
We have the natural identifications
\begin{align*}
\Lambda^k Z\otimes \Lambda^{>(m-k)(n+1)} W\otimes U &\cong P^{k,\geq (m-k)(n+1)+1}, \\
\Lambda^k Z\otimes (\Lambda^{>(m-k)(n+1)} W/\Lambda^{>(m-k)(n+1)+n} W)\otimes U &\cong P^{k,[(m-k)(n+1)+1, (m-k)(n+1)+n]}.
\end{align*}
Thus, applying $\Lambda^k Z\otimes -\otimes U$ to the preceding diagram gives the following commutative diagram:
$$
\xymatrix{
M_k \ar[rd] \ar[dd]_{\simeq} \ar[rr]^{\text{id}} & & M_k\\
& Q_k \ar[ur] \ar[d]_{\simeq} \\
P^{k,\geq (m-k)(n+1)+1} \ar[r]^(0.4){\text{pr}} & P^{k,[(m-k)(n+1)+1, (m-k)(n+1)+n]}
}
$$

Since $M_k$ and $Q_k$ admit trivial $\Lambda^+ Z$-actions, they can also be regarded as $\Lambda V$-modules. Hence, we can choose surjective $\Lambda V$-semifree resolutions
$$
\rho: F(P^{k,\geq(m-k)(n+1)+1}) \stackrel{\simeq}{\longrightarrow} M_k, \quad \rho': F(P^{k,[(m-k)(n+1)+1,(m-k)(n+1)+n]}) \stackrel{\simeq}{\longrightarrow} Q_k.
$$
The composites of the left and middle vertical arrows in the diagram below are then $\Lambda V$-semifree resolutions. By Proposition \ref{commutative lift}(ii), there exist lifts $f$ and $g$ that make the diagram commute:
$$
\xymatrix{
F(P^{k,\geq(m-k)(n+1)+1}) \ar@{-->}[r]^(0.4){f} \ar[d]_{\rho}^{\simeq} & F(P^{k,[(m-k)(n+1)+1,(m-k)(n+1)+n]}) \ar@{-->}[r]^(0.6){g} \ar[d]_{\rho'}^{\simeq} & F(P^{k,\geq(m-k)(n+1)+1}) \ar[d]_{\rho}^{\simeq} \\
M_k \ar[r] \ar[d]^{\simeq} & Q_k \ar[r] \ar[d]^{\simeq} & M_k \\
P^{k,\geq(m-k)(n+1)+1} \ar[r]^(0.4){\text{pr}} & P^{k,[(m-k)(n+1)+1,(m-k)(n+1)+n]}
}
$$

Since $\rho\circ(g\circ f)=\text{id}\circ \rho=\rho$ and $\rho$ is a quasi-isomorphism, Proposition \ref{commutative lift}(i) implies that $g\circ f\sim \text{id}$. Indeed, we can construct a graded $\Lambda V$-linear endomorphism $\theta$ on $F(P^{k,\geq(m-k)(n+1)+1})$ of degree $-1$ such that
$$
g\circ f-\text{id}=d\theta+\theta d, \quad \rho\circ\theta=0.
$$
Write the semifree module $F(P^{k,\geq(m-k)(n+1)+1})$ as the union of submodules
$$
F(P)(j)=\Lambda V\otimes E(j)
$$
with $F(P)(j)/F(P)(j-1)$ a free DG module over $\Lambda V$, $dE(j)\subset F(P)(j-1)$, and, in particular, $F(P)(-1)=0$. Assume that $\theta$ has been constructed on $F(P)(j-1)$ and satisfies the required conditions. Choose a basis of $E(j)/E(j-1)$ and, for each basis element, choose a representative in $E(j)$. Let $x$ be such a representative. Then $dx\in F(P)(j-1)$, and by the induction hypothesis
\begin{align*}
& d(g\circ f-\text{id}-\theta d)(x) = (g\circ f\circ d)(x)-dx-d\theta(dx) = (g\circ f-\text{id}-d\theta)(dx) =\theta d(dx)=0, \\
& \rho(g\circ f-\text{id}-\theta d)(x) = (\rho\circ g\circ f)(x)-\rho(x) = 0.
\end{align*}
Hence, $(g\circ f-\text{id}-\theta d)(x)$ is a cocycle in $\ker\rho$. Since $\rho$ is a surjective quasi-isomorphism, $\ker\rho$ is acyclic. Consequently, there exists $y\in\ker\rho$ such that $dy=(g\circ f-\text{id}-\theta d)(x)$. Set $\theta(x)=y$. Then $(g\circ f-\text{id})(x)=(d\theta+\theta d)(x)$ and $(\rho\circ\theta)(x)=0$. By graded $\Lambda V$-linearity, the definition of $\theta$ can be extended to $F(P)(j)$, and we can define $\theta$ on $F(P^{k,\geq(m-k)(n+1)+1})$ inductively.

By Proposition \ref{replace homotopy} (i) and (ii), there exist a surjective quasi-isomorphism
$$
\widetilde{\rho}:N \stackrel{\simeq}{\longrightarrow} F(P^{k,[(m-k)(n+1)+1,(m-k)(n+1)+n]})
$$
and morphisms
$$
F: F(P^{k,\geq(m-k)(n+1)+1}) \to N, \quad G: N \to F(P^{k,\geq(m-k)(n+1)+1})
$$
such that
$$
\widetilde{\rho} \circ F=f, \quad G\circ F = \text{id}.
$$
We verify that the replacement of $F(P^{k,[(m-k)(n+1)+1,(m-k)(n+1)+n]}),\rho',f,g$ by $N,\rho'\circ\widetilde{\rho},F,G$, respectively, still makes the above diagram commute. For the square on the upper-left, the composite of $\rho$ with $M_k\to Q_k$ is $\rho'\circ f=\rho'\circ\widetilde{\rho} \circ F$. Hence, this square commutes. For the square on the upper-right, recall that
$$
N = F(P^{k,\geq(m-k)(n+1)+1}) \oplus F(P^{k,[(m-k)(n+1)+1,(m-k)(n+1)+n]}) \oplus SF(P^{k,\geq(m-k)(n+1)+1}).
$$
First consider the restrictions to $F(P^{k,\geq(m-k)(n+1)+1})$. The restrictions of $\widetilde{\rho}$ and $G$ become $f$ and $\text{id}$, respectively. The composite of $f$, $\rho'$ and $Q_k\to M_k$ is $\rho\circ g\circ f$. The difference $\rho\circ g\circ f-\rho\circ\text{id} = \rho\circ(d\theta+\theta d)$. By the construction of $\theta$ in the preceding paragraph, we have $\im \theta\subset\ker\rho$. Thus, this composite agrees with $\rho\circ\text{id}$. Next consider the restrictions to $F(P^{k,[(m-k)(n+1)+1,(m-k)(n+1)+n]})$. The restrictions of $\widetilde{\rho}$ and $G$ become $-\text{id}$ and $-g$, respectively. The composite of $-\text{id}$, $\rho'$ and $Q_k\to M_k$ is $-\rho\circ g$, which agrees with $\rho\circ (-g)$. Finally, consider the restrictions to $SF(P^{k,\geq(m-k)(n+1)+1})$. The restriction of $\widetilde{\rho}$ becomes $0$. As for $G$, the image of its restriction is in $\im \theta \subset \ker\rho$. This completes the proof that this replacement still makes the diagram commute. Moreover, we have $G\circ F=\text{id}$. Therefore, without loss of generality, we may assume $g\circ f=\text{id}$.

On the other hand,
$$
P^{k,\geq (m-k)(n+1)+n+1} = P^{k,\geq (m+1-k)(n+1)} \cong I_k/I_{k+1}
$$
is a subspace of $P/I_{k+1}$. Hence,
$$
P^{k,[(m-k)(n+1)+1,(m-k)(n+1)+n]} \cong \frac{P^{\geq k,\geq (m-k)(n+1)+1}+I_{k+1}}{I_k}
$$
can be identified with a subspace of $P/I_k$. It is in fact a $\Lambda V$-submodule, as both the differential and the $\Lambda V$-action preserve it. We thus obtain the following commutative diagram of $\Lambda V$-modules:
$$
\xymatrix{
F(P^{k,\geq(m-k)(n+1)+1}) \ar@{-->}[rr]^(0.6){h} \ar[rd]^{\simeq} \ar[dd]_f & & P_{k+1} \ar[rd]^{\simeq} \\
& P^{k,\geq(m-k)(n+1)+1} \ar@{^{(}->}[rr]^(0.6){\iota} \ar[dd]_{\text{pr}} & & P/I_{k+1} \ar[dd]_{\text{pr}} \\
F(P^{k,[(m-k)(n+1)+1,(m-k)(n+1)+n]}) \ar[rd]^{\simeq} \\
& P^{k,[(m-k)(n+1)+1,(m-k)(n+1)+n]} \ar@{^{(}->}[rr]^(0.65){\iota} & & P/I_k 
}
$$
Here $\iota$ denotes the inclusions, and $h$ is a lift obtained from Proposition \ref{commutative lift}(ii). By Proposition \ref{replace homotopy}(i), $h$ can be factored as an injective morphism $H:F(P^{k,\geq(m-k)(n+1)+1}) \to \widetilde{P}_{k+1}$ followed by a surjective quasi-isomorphism $\widetilde{P}_{k+1} \stackrel{\simeq}{\longrightarrow} P_{k+1}$. Replacing $P_{k+1}$ and $h$ by $\widetilde{P}_{k+1}$ and $H$, respectively, we may assume that $h$ is injective  without loss of generality.

Let $(F(P^{k,[(m-k)(n+1)+1,(m-k)(n+1)+n]})\oplus P_{k+1})/\sim$ denote the pushout of $f$ and $h$, where the submodule generated by $(f(x),-h(x))$ is identified with 0 for $x\in F(P^{k,\geq(m-k)(n+1)+1})$. Define
$$
\phi: F(P^{k,[(m-k)(n+1)+1,(m-k)(n+1)+n]}) \stackrel{\simeq}{\longrightarrow} P^{k,[(m-k)(n+1)+1,(m-k)(n+1)+n]} \stackrel{\iota}{\longrightarrow} P/I_k,
$$
and
$$
\psi: P_{k+1} \stackrel{\simeq}{\longrightarrow} P/I_{k+1} \stackrel{\text{pr}}{\longrightarrow} P/I_k.
$$
From the preceding commutative diagram, the morphism
$$
F(P^{k,[(m-k)(n+1)+1,(m-k)(n+1)+n]})\oplus P_{k+1} \to P/I_k, \quad (a,b)\mapsto \phi(a)+\psi(b)
$$
vanishes on $(f(x),-h(x))$. Hence, it descends to a morphism
$$
(\phi,\psi): (F(P^{k,[(m-k)(n+1)+1,(m-k)(n+1)+n]})\oplus P_{k+1})/\sim \to P/I_k.
$$
We claim that $(\phi,\psi)$ is a quasi-isomorphism.

The proof of this claim is the dual analogue of Lemma 13.3 in \cite{FHT}. As $h$ and $\iota: P^{k,\geq(m-k)(n+1)+1} \to P/I_{k+1}$ are injective, we have the following commutative diagram of short exact sequences:
$$
\xymatrix{
0 \ar[r] & F(P^{k,\geq(m-k)(n+1)+1}) \ar[r]^(0.68){h} \ar[d]_{\simeq} & P_{k+1} \ar[r] \ar[d]_{\simeq} & \text{coker\,} h \ar[r] \ar[d] & 0 \\
0 \ar[r] & P^{k,\geq(m-k)(n+1)+1} \ar[r]^(0.6){\iota} & P/I_{k+1} \ar[r] & \text{coker\,} \iota \ar[r] & 0
}
$$
By the five Lemma, the induced morphism $\text{coker\,} h \to \text{coker\,} \iota$ is a quasi-isomorphism as well.

Observe that the morphism
$$
\chi: F(P^{k,[(m-k)(n+1)+1,(m-k)(n+1)+n]}) \to (F(P^{k,[(m-k)(n+1)+1,(m-k)(n+1)+n]})\oplus P_{k+1})/\sim, \quad a\mapsto [(a,0)]
$$
is injective, with cokernel $\text{coker\,} h$. Indeed, if $(a_1,0)\sim(a_2,0)$, then there exists some $x\in F(P^{k,\geq(m-k)(n+1)+1})$ such that
$$
f(x)=a_1-a_2, \quad -h(x)=0.
$$
As $h$ is injective, we have $x=0$ and $a_1=a_2$. Thus, $\chi$ is injective. Its image is $(F(P^{k,[(m-k)(n+1)+1,(m-k)(n+1)+n]})\oplus \text{im\,} h)/\sim$. Consequently, its cokernel is
$$
(F(P^{k,[(m-k)(n+1)+1,(m-k)(n+1)+n]})\oplus P_{k+1})/(F(P^{k,[(m-k)(n+1)+1,(m-k)(n+1)+n]})\oplus \text{im\,} h) \cong \text{coker\,} h.
$$

On the other hand, $\text{coker\,} \iota$ is also the cokernel of the inclusion
$$\iota:P^{k,[(m-k)(n+1)+1,(m-k)(n+1)+n]} \to P/I_k.
$$
Indeed,
\begin{align*}
\text{coker\,} \iota &= (P/I_{k+1})/P^{k,\geq(m-k)(n+1)+1} \\
& \cong (P/I_{k+1})/((P^{\geq k,\geq(m-k)(n+1)+1}+I_{k+1})/I_{k+1}) \\
& \cong (P/I_k)/((P^{\geq k,\geq(m-k)(n+1)+1}+I_{k+1})/I_k) \\
& \cong (P/I_k)/P^{k,[(m-k)(n+1)+1,(m-k)(n+1)+n]}.
\end{align*}

The composite $(\phi,\psi)\circ\chi$ is precisely $\phi$. Hence, we obtain the following commutative diagram of short exact sequences:
$$
\xymatrix{
0 \ar[r] & F(P^{k,[(m-k)(n+1)+1,(m-k)(n+1)+n]}) \ar[r]^(0.42){\chi} \ar[d]_{\simeq} & (F(P^{k,[(m-k)(n+1)+1,(m-k)(n+1)+n]})\oplus P_{k+1})/\sim \ar[r] \ar[d]_{(\phi,\psi)} & \text{coker\,} h \ar[r] \ar[d]_{\simeq} & 0 \\
0 \ar[r] & P^{k,[(m-k)(n+1)+1,(m-k)(n+1)+n]} \ar[r]^(0.6){\iota} & P/I_k \ar[r] & \text{coker\,} \iota \ar[r] & 0
}
$$
A second application of the five lemma shows that $(\phi,\psi)$ is a quasi-isomorphism. This proves the claim.

Now define a morphism
$$
F(P^{k,[(m-k)(n+1)+1,(m-k)(n+1)+n]})\oplus P_{k+1} \to P_{k+1}, \quad (a,b)\mapsto (h\circ g)(a)+b.
$$
For $x\in F(P^{k,\geq(m-k)(n+1)+1})$, the image of $(f(x),-h(x))$ is
$$(h\circ g\circ f)(x)-h(x)=h(x)-h(x)=0.
$$
Therefore, it induces a morphism
$$
(F(P^{k,[(m-k)(n+1)+1,(m-k)(n+1)+n]})\oplus P_{k+1})/\sim \to P_{k+1},
$$
which by definition sends $[(0,b)]$ to $b$. Thus the identity map on $P_{k+1}$ factors as
$$
P_{k+1} \to (F(P^{k,[(m-k)(n+1)+1,(m-k)(n+1)+n]})\oplus P_{k+1})/\sim \to P_{k+1}, \quad b\mapsto [(0,b)] \mapsto b.
$$
Moreover, $(\phi,\psi)[(0,b)]=\psi(b)$.

Choose a surjective $\Lambda V$-semifree resolution
$$
P_k \stackrel{\simeq}{\longrightarrow} (F(P^{k,[(m-k)(n+1)+1,(m-k)(n+1)+n]})\oplus P_{k+1})/\sim.
$$
Then we obtain a lift $P_{k+1}\to P_k$ rendering the following diagram commutative:
$$
\xymatrix{
P_{k+1} \ar@{-->}[rd] \ar[rdd] \ar[rr]^{\text{id}} \ar[ddd]_{\simeq} & & P_{k+1} \\
& P_k \ar[d]_{\simeq} \\
& (F(P^{k,[(m-k)(n+1)+1,(m-k)(n+1)+n]})\oplus P_{k+1})/\sim \ar[uur] \ar[d]_{\simeq}^{(\phi,\psi)} \\
P/I_{k+1} \ar[r]^{\text{pr}} & P/I_k
}
$$

Therefore, by induction we obtain $P_i\stackrel{\simeq}{\longrightarrow} P/I_i$ for $i\leq m+1$, with $P_{i+1}\to P_i\to P_{i+1}$ factoring the identity map. This yields the following commutative diagram:
$$
\xymatrix{
P \ar[rd]^{\xi} \ar[rdd]_{\text{pr}} \ar[rr]^{\text{id}} & & P \\
& P_0 \ar[ur] \ar[d]_{\eta}^{\simeq} \\
& P/I_0
}
$$

Finally, since $\Lambda^{\geq (m+1)(n+1)} V\cdot P \subset I_0$, the projection $P\to P/I_0$ factors through $P/(\Lambda^{>(m+1)(n+1)-1} V\cdot P)$. Choose a surjective $\Lambda V$-semifree resolution
$$
Q \stackrel{\simeq}{\longrightarrow} P/(\Lambda^{>(m+1)(n+1)-1} V\cdot P).
$$
We then have the following commutative diagrams with lifts $\zeta$ and $\lambda$.
$$
\xymatrix{
P \ar@{-->}[r]^{\zeta} \ar[rd]_{\text{pr}} & Q \ar@{-->}[r]^{\lambda} \ar[d]_{\simeq} & P_0 \ar[d]_{\eta}^{\simeq} \\
& P/(\Lambda^{>(m+1)(n+1)-1} V\cdot P) \ar[r]^(0.75){\text{pr}} & P/I_0
}
$$
Hence, $\eta\circ(\lambda\circ\zeta)=\eta\circ\xi$. Since $\eta$ is a quasi-isomorphism, Proposition \ref{commutative lift}(i) gives $\lambda\circ\zeta\sim\xi$. Applying Proposition \ref{replace homotopy}(ii) again, we can replace $Q$ by a larger module, so that $\lambda\circ\zeta=\xi$ strictly. The resulting commutative diagram is
$$
\xymatrix{
P \ar[rd]^{\zeta} \ar[rdd]_{\text{pr}} \ar[rrr]^{\text{id}} & & & P \\
& Q \ar[r]^{\lambda} \ar[d]_{\simeq} & P_0 \ar[ur] \\
& P/(\Lambda^{>(m+1)(n+1)-1} V\cdot P)
}
$$
Therefore,
$$
\text{mcat}_{\Lambda V}(P)\leq (m+1)(n+1)-1.
$$
\end{proof}

When $\Lambda V$ itself is a minimal Sullivan algebra and $n\geq 1$, we have the sharper inclusion
$$
dP^{\geq p,\geq q} \subset P^{\geq p+2,\geq q-1}+P^{\geq p,\geq q}.
$$
In this case we can choose a larger DG submodule $I_k$. Specifically, set
$$
I_k:=I_{k+1}+P^{\geq k,\geq(m+1-k)n+1}.
$$
Then the subspace $P^{\geq k,\geq(m-k)n+1}+I_{k+1}$ is closed under the differential. It follows that $(P^{\geq k,\geq(m-k)n+1}+I_{k+1})/I_{k+1}$ is a DG module over $\Lambda W$, and its submodule $\Lambda^n W \cdot [(P^{\geq k,\geq(m-k)n+1}+I_{k+1})/I_{k+1}]$ is isomorphic to $I_k/I_{k+1}$. Following the construction in Theorem \ref{main}, we can factor $\text{id}_P$ as
$$
P \to F(P/(\Lambda^{\geq (m+1)n+1} V\cdot P)) \to F(P/I_0) \to P.
$$
We therefore obtain the following estimate.

\begin{thm}\label{main for minimal}
Under the hypothesis of Theorem \ref{main}, suppose in addition that $\Lambda V$ itself is a minimal Sullivan algebra. If ${\rm mcat}_{\Lambda Z}(\overline{P})=m$ and ${\rm cat}(\Lambda W)=n\geq 1$, then
$$
{\rm mcat}_{\Lambda V}(P)\leq (m+1)n.
$$
\end{thm}

A minimal Sullivan algebra $\Lambda V$ can be regarded as a relative Sullivan algebra $\Lambda V^1\otimes\Lambda V^{\geq 2}$. Consequently, $\text{cat}(\Lambda V)$ is finite whenever both $\text{cat}(\Lambda V^1)$ and $\text{cat}(\Lambda V^{\geq 2})$ are finite. Conversely, if in addition $\dim V^i<\infty$ for $i\geq 2$, Proposition 9.6 and Theorem 9.3 of \cite{FHT2} give the reverse implication. Hence, we have obtained Corollary \ref{cat for minimal intro}.

\begin{cor}
Let $\Lambda V$ be a minimal Sullivan algebra with $\dim V^i<\infty$ for all $i\geq 2$. Then ${\rm cat}(\Lambda V)<\infty$ if and only if both ${\rm }cat(\Lambda V^1)$ and ${\rm cat}(\Lambda V^{\geq 2})$ are finite.
\end{cor}

The following two simple examples show that this corollary cannot be generalized to arbitrary relative Sullivan algebras.

\begin{ex}
Let $Z$ be spanned by $z$ of degree 3, and $W$ be spanned by $w$ of degree 2. Set $dz=0,dw=z$. Then $\text{cat}(\Lambda Z\otimes\Lambda W)=0$ but $\text{cat}(\Lambda W)=\infty$.
\end{ex}

Although $\text{cat}(\Lambda W)\leq \text{cat}(\Lambda Z\otimes\Lambda W)$ if $\Lambda Z\otimes\Lambda W$ is a minimal Sullivan algebra \cite[Theorem 9.3]{FHT2}, it is still possible that $\text{cat}(\Lambda Z)=\infty$ but $\text{cat}(\Lambda Z\otimes\Lambda W)$ is finite.

\begin{ex}
Let $Z$ be spanned by $z$ of degree 2, and $W$ be spanned by $w$ of degree 3. Set $dz=0,dw=z^2$. Then $\Lambda Z\otimes\Lambda W$ is minimal, and its LS category is 1, while $\text{cat}(\Lambda Z)=\infty$.
\end{ex}

For any fixed $m$ and $n$, there are examples such that $cat(\Lambda Z)=m$, $cat(\Lambda W)=n$ and $cat(\Lambda Z\otimes\Lambda W)=(m+1)(n+1)-1$.

\begin{ex}[Example 1 of \cite{FHT}, Chapter 30(b)]
Let $Z$ be spanned by $x,y$, and $W$ be spanned by $u,v$, with degrees
$$
|x|=2(n+1), \quad |y|=2(m+1)(n+1)-1, \quad |u|=2, \quad |v|=2(n+1)-1.
$$
Set
$$
dx=du=0, \quad dy=x^{m+1}, \quad dv=u^{n+1}-x.
$$
Then $\text{cat}(\Lambda Z)=m$, and $\text{cat}(\Lambda W)=n$. Moreover, $\Lambda Z\otimes\Lambda W$ is quasi-isomorphic to $\Lambda (u)/(u^{(m+1)(n+1)})$, whose LS category is $(m+1)(n+1)-1$.
\end{ex}

This example shows that the bound in Theorem \ref{main} is sharp. It remains open whether the bound in Theorem \ref{main for minimal} can be improved when $m\geq 1$ and $n\geq 2$. The case $m=0$ is trivial. Indeed, we can take $\Lambda Z=\mathbbm{k}$. For an arbitrary minimal Sullivan algebra $\Lambda W$ with $\text{cat}(\Lambda W)=n$, we have
$$
\text{cat}(\Lambda Z\otimes\Lambda W)=\text{cat}(\Lambda W)=n=(m+1)n.
$$

The sharpness of the case $n=1$ is demonstrated by the following example, which was given by the referee. Let $Z$ be spanned by $z_1,\ldots,z_m$ with $|z_i|=1$ for each $i$, and let $W$ be spanned by $w$ with $|w|=3$. Set differential of $\Lambda Z\otimes\Lambda W$ as 0. By Example 6 in Chapter 29(b) of \cite{FHT}, we have
$$
\text{cat}(\Lambda Z)=m, \quad \text{cat}(\Lambda W)=1, \quad \text{cat}(\Lambda Z\otimes\Lambda W)=m+1=(m+1)n.
$$

\subsection{Estimating the Toomer invariant}
The idea of Theorem \ref{main} can also be used to estimate the Toomer invariant of relative Sullivan algebras, although the hypotheses and the resulting bound differ. The proof in this case is simpler, as we can apply the four lemma.

\begin{lem}[Four lemma]
In the following commutative diagram of abelian groups with exact rows,
$$
\xymatrix{
A \ar[r] \ar[d]_{\xi} & B \ar[r] \ar[d]_{\eta} & C \ar[r] \ar[d]_{\zeta} & D \ar[d]_{\lambda} \\
A' \ar[r] & B' \ar[r] & C' \ar[r] & D'
}
$$
if $\eta,\lambda$ are injective and $\xi$ is surjective, then $\zeta$ is injective.
\end{lem}

\begin{thm}\label{estimate e}
Let $\Lambda V=\Lambda Z\otimes\Lambda W$ be a relative Sullivan algebra, and $P$ a $\Lambda V$-semifree module. Regard $P$ as a $\Lambda Z$-module and assume that
$$
e_{\Lambda Z}(P)=m.
$$
For each $p,q$, regard $P^{p,\geq q}$ as a $\Lambda W$-module as in the proof of Theorem \ref{main}. Suppose there exists a fixed integer $n\geq 0$ such that the inclusion
$$
P^{p,\geq q+n+1} \to P^{p, \geq q}
$$
induces the zero map on cohomology for all $p,q$. Then
$$
e_{\Lambda V}(P)\leq (m+1)(n+2)-2.
$$
Moreover, if $\Lambda V$ itself is a minimal Sullivan algebra, then
$$
e_{\Lambda V}(P)\leq (m+1)(n+1)-1.
$$
\end{thm}
\begin{proof}
First observe that $P$ is $\Lambda Z$-semifree. Since $e_{\Lambda Z}(P)=m$, the projection $P \to P/(\Lambda^{>m} Z\cdot P)$ is injective on cohomology. Analogous to the Proof of Theorem \ref{main}, define
$$
I_{m+1}:=P^{\geq m+1,\geq 0}
$$
and, inductively for $0\leq k\leq m$,
$$
I_k:=I_{k+1}+P^{\geq k,\geq (m+1-k)(n+2)-1},
$$
or
$$
I_k:=I_{k+1}+P^{\geq k,\geq (m+1-k)(n+1)}
$$
if $\Lambda V$ is minimal.

Consider the short exact sequences
$$
0 \to P^{k,\geq (m-k)(n+2)+n+1} \to P^{k,\geq (m-k)(n+2)} \to P^{k,[(m-k)(n+2), (m-k)(n+2)+n]} \to 0,
$$
and
$$
0 \to P^{k,\geq (m-k)(n+1)+n+1} \to P^{k,\geq (m-k)(n+1)} \to P^{k,[(m-k)(n+1), (m-k)(n+1)+n]} \to 0.
$$
By hypothesis, the inclusions induce zero maps on cohomology. Consequently, the projections induce injections on cohomology.

In analogy with the proof of Theorem \ref{main}, we have
$$
P^{k,[(m-k)(n+2), (m-k)(n+2)+n]} \cong \frac{P^{\geq k,\geq (m-k)(n+2)}+I_{k+1}}{I_k}
$$
which is a submodule of $P/I_k$. On the other hand,
$$
P^{k,\geq (m-k)(n+2)} = \frac{P^{\geq k,\geq (m-k)(n+2)}+I_{k+1}}{I_{k+1}}.
$$
Hence, we obtain the following commutative diagram of short exact sequences:
$$
\xymatrix{
0 \ar[r] & P^{k,\geq (m-k)(n+2)} \ar[r] \ar[d]_{\text{pr}} & P/I_{k+1} \ar[r] \ar[d]_{\text{pr}} & \frac{P}{P^{\geq k,\geq (m-k)(n+2)}+I_{k+1}} \ar[r] \ar[d]_{\simeq} & 0 \\
0 \ar[r] & P^{k,[(m-k)(n+2), (m-k)(n+2)+n]} \ar[r] & P/I_k \ar[r] & \frac{P}{P^{\geq k,\geq (m-k)(n+2)}+I_{k+1}} \ar[r] & 0
}
$$

In the minimal case, the corresponding submodule of $P/I_k$ is
$$
P^{k,[(m-k)(n+1), (m-k)(n+1)+n]} \cong \frac{P^{\geq k,\geq (m-k)(n+1)}+I_{k+1}}{I_k},
$$
because
$$
dP^{\geq k,\geq (m-k)(n+1)} \subset P^{\geq k+2,\geq (m-k)(n+1)-1}+P^{\geq k,\geq (m-k)(n+1)} \subset I_{k+1}+P^{\geq k,\geq (m-k)(n+1)}.
$$
On the other hand,
$$
P^{k,\geq (m-k)(n+1)} = \frac{P^{\geq k,\geq (m-k)(n+1)}+I_{k+1}}{I_{k+1}}.
$$
The commutative diagram is
$$
\xymatrix{
0 \ar[r] & P^{k,\geq (m-k)(n+1)} \ar[r] \ar[d]_{\text{pr}} & P/I_{k+1} \ar[r] \ar[d]_{\text{pr}} & \frac{P}{P^{\geq k,\geq (m-k)(n+1)}+I_{k+1}} \ar[r] \ar[d]_{\simeq} & 0 \\
0 \ar[r] & P^{k,[(m-k)(n+1), (m-k)(n+1)+n]} \ar[r] & P/I_k \ar[r] & \frac{P}{P^{\geq k,\geq (m-k)(n+1)}+I_{k+1}} \ar[r] & 0
}
$$

In both cases, the left vertical morphisms induce injections on cohomology, and the right vertical morphisms are quasi-isomorphisms. By the four lemma, the middle vertical morphisms also induce injections on cohomology. Therefore, we obtain a sequence of projections
$$
P \to P/I_{m+1} \to P/I_m \to \ldots \to P/I_1 \to P/I_0
$$
all of which are injective on cohomology. The composite of these projections factors as
$$
P \to P/(\Lambda^{\geq (m+1)(n+2)-1} V\cdot P) \to P/I_0
$$
(respectively, $P \to P/(\Lambda^{\geq (m+1)(n+1)} V\cdot P) \to P/I_0$), which implies that the projection
$$
P \to P/(\Lambda^{\geq (m+1)(n+2)-1} V\cdot P)
$$
(respectively, $P \to P/(\Lambda^{\geq (m+1)(n+1)} V\cdot P)$) is injective on cohomology. Therefore,
$$
e_{\Lambda V}(P) \leq (m+1)(n+2)-2
$$
(respectively, $e_{\Lambda V}(P) \leq (m+1)(n+1)-1$).
\end{proof}

When $P=\Lambda V$, we have
$$
P^{p,\geq q} \cong (\Lambda^p Z,d_0)\otimes(\Lambda^{\geq q} W,\bar{d}).
$$
Its cohomology is $H^*(\Lambda^p Z,d_0)\otimes H^*(\Lambda^{\geq q} W,\bar{d})$. Consequently, the condition that, for some fixed $n\geq 0$,
all maps
$$
H^*(P^{p,\geq q+n+1}) \to H^*(P^{p, \geq q}) \quad (p,q\geq 0)
$$
are zero can be replaced by the condition that, for the same $n$, all maps
$$
H^*(\Lambda^{\geq q+n+1} W)\to H^*(\Lambda^{\geq q} W) \quad (q\geq 0)
$$
are zero.

Moreover, this vanishing condition follows from $\text{cat}(\Lambda W)=n$ by Lemma \ref{ideal factorization}. In particular, taking $q=0$ yields $e(\Lambda W)\leq n$. Thus, we obtain the implications
$$
\text{cat}(\Lambda W)\leq n \quad \Rightarrow \quad H^*(\Lambda^{\geq q+n+1} W)\to H^*(\Lambda^{\geq q} W) \text{ are zero for all } q\geq 0 \quad \Rightarrow \quad e(\Lambda W)\leq n.
$$

The first condition is strictly stronger than the second. Indeed, there exist examples with $e(\Lambda W)<\text{cat}(\Lambda W)=n$ (see \cite{LemaireS}, also Example 2 of \cite{FHT}, Chapter 29(b)). For such an example, $n-1$ satisfies the second condition, while cat$(\Lambda W)>n-1$. To see this, the case $q=0$ follows from $e(\Lambda W)<n$, and the case $q\geq 1$ follows from cat$(\Lambda W)=n$ and Lemma \ref{ideal factorization}.

In fact, $\text{cat}(\Lambda W)=n$ yields sharper bounds. More generally, let $P$ be a $\Lambda V$-semifree module of the form $\Lambda V\otimes U$, where $U=\bigcup_k U(k)$ is an increasing union of subspaces with
$$
dU(k) \subset \Lambda Z\otimes U(k-1),
$$
and put $e_{\Lambda Z}(P)=m$. (In other words, these are the hypotheses of Theorem \ref{main}, except that the condition $\text{mcat}_{\Lambda Z}(\overline{P})=m$ is replaced by the weaker condition $e_{\Lambda Z}(P)=m$.) As in the proof of that theorem, set
$$
I_{m+1}:=P^{\geq m+1,\geq 0}
$$
and
$$
I_k:=I_{k+1}+P^{\geq k,\geq (m+1-k)(n+1)},
$$
or, if $\Lambda V$ is minimal and $n\geq 1$,
$$
I_k:=I_{k+1}+P^{\geq k,\geq (m+1-k)n+1}.
$$
Choose an arbitrary surjective semifree resolution $P_{m+1} \stackrel{\simeq}{\longrightarrow} P/I_{m+1}$ and construct $P_k$ successively until $P_0$ is obtained. The same argument shows that, for $0\leq k\leq m$, each morphism $P_{k+1}\to P_k$ induces an injection on cohomology (indeed, $\text{id}_{P_{k+1}}$ factors through $P_k$). This is equivalent to saying that $P/I_{k+1} \to P/I_k$ induces an injection on cohomology. On the other hand, unlike in the proof of Theorem \ref{main}, no retraction of $P\to P_{m+1}$ is required. Since $e_{\Lambda Z}(P)=m$, the projection $P\to P/I_{m+1}$ is injective on cohomology. Composing $P\to P/I_{m+1}$ with morphisms $P/I_{k+1} \to P/I_k$ gives a morphism $P\to P/I_0$, which is also injective on cohomology. Moreover, this morphism factors through $P/(\Lambda^{>(m+1)(n+1)-1}V\cdot P)$ (respectively, $P/(\Lambda^{>(m+1)n}V\cdot P)$). This proves Theorem \ref{estimate e intro}.

\bibliographystyle{newplainnat}
\bibliography{refs}

\begin{thebibliography}{1}

\bibitem{FHT}
Y.~F{\'e}lix, S.~Halperin, and J.-C. Thomas.
\newblock {\em Rational homotopy theory}, volume 205 of {\em Graduate Texts in
  Mathematics}.
\newblock Springer-Verlag, New York, 2001.

\bibitem{FHT2}
Y.~F{\'e}lix, S.~Halperin, and J.-C. Thomas.
\newblock {\em Rational homotopy theory II}.
\newblock World Scientific Publishing Co. Pte. Ltd., Hackensack, NJ, 2015.

\bibitem{HL}
S.~Halperin and J.-M. Lemaire.
\newblock Notions of category in differential algebra.
\newblock In {\em Algebraic topology---rational homotopy ({L}ouvain-la-{N}euve,
  1986)}, volume 1318 of {\em Lecture Notes in Math.}, pages 138--154.
  Springer, Berlin, 1988.

\bibitem{Halperin}
S.~Halperin.
\newblock Lectures on minimal models.
\newblock {\em M\'{e}m. Soc. Math. France (N.S.)}, (9-10):261, 1983.

\bibitem{hess}
K.~P. Hess.
\newblock A proof of {G}anea's conjecture for rational spaces.
\newblock {\em Topology}, 30(2):205--214, 1991.

\bibitem{LemaireS}
J.-M. Lemaire and F.~Sigrist.
\newblock Sur les invariants d'homotopie rationnelle li\'{e}s \`a la {L}. {S}.
  cat\'{e}gorie.
\newblock {\em Comment. Math. Helv.}, 56(1):103--122, 1981.

\bibitem{LS}
L.~Ljusternik and L.~Schnirelmann.
\newblock {\em M{\'e}thodes topologiques dans les probl{\`e}mes variationnels:
  Espaces {\`a} un nombre fini de dimensions. 1{\`e}re partie}, volume 188.
\newblock Hermann, 1934.

\bibitem{sullivan}
D.~Sullivan.
\newblock Infinitesimal computations in topology.
\newblock {\em Publications Math{\'e}matiques de l'IH{\'E}S}, 47:269--331,
  1977.

\bibitem{weibel}
C.~A. Weibel.
\newblock {\em An introduction to homological algebra. 1st pbk-ed}, volume~38
  of {\em Camb. Stud. Adv. Math.}
\newblock Cambridge: Cambridge Univ. Press, 1st pbk-ed. edition, 1995.

\end{thebibliography}

\vskip 1 cm
\noindent
{Department of Mathematics, Nanchang University, Nanchang, Jiangxi, 330031, China}\\
{\it Email address:}~{\tt jiaweizhou90@ncu.edu.cn}
\end{document}